\documentclass{gtart}
\usepackage{amsmath, amssymb, graphicx, epsfig, verbatim, graphs}
\usepackage{amscd}
\usepackage{amssymb}
\newcommand{\C}{{\mathbb C}}

\newtheorem{thm}{Theorem}[section]

\newtheorem{cor}[thm]{Corollary}
\newtheorem{lem}[thm]{Lemma}
\newtheorem{prop}[thm]{Proposition}
\newtheorem{defn}[thm]{Definition}

\newtheorem{rem}[thm]{Remark}
\newtheorem{rems}[thm]{Remarks}

\numberwithin{equation}{section}

\def\G{\Gamma}

\newcommand{\bfz}{{\mathbb {Z}}}

\newcommand{\bfq}{{\mathbb {Q}}}
\newcommand{\J}{{\mathcal {J}}}

\newcommand{\cpkk}{{\overline {{\mathbb C}{\mathbb P}^2}}}
\newcommand{\cpk}{{\mathbb {CP}}^2}
\newcommand{\cp}{{\mathbb {CP}}}

\newcommand{\frg}{{\mathcal {G}}}

\newcommand{\fra}{{\mathcal {A}}}
\newcommand{\frb}{{\mathcal {B}}}
\newcommand{\frc}{{\mathcal {C}}}
\newcommand{\frw}{{\mathcal {W}}}

\newcommand{\frn}{{\mathcal {N}}}
\newcommand{\farm}{{\mathcal {M}}}

\begin{document}

\title{Weighted homogeneous singularities and rational homology disk
  smoothings}

\author{Mohan Bhupal}
\address{Middle East Technical University, Ankara, Turkey}
\email{bhupal@metu.edu.tr}

\author[Andr{\'a}s I. Stipsicz]{Andr{\'a}s I. Stipsicz}
\address{R\'enyi Institute of Mathematics, Budapest, Hungary}
\email{stipsicz@renyi.hu}



\begin{abstract}
We classify the resolution graphs of weighted homogeneous surface
singularities which admit rational homology disk smoothings. The
nonexistence of rational homology disk smoothings is shown by
symplectic geometric methods, while the existence is verified via
smoothings of negative weight.
\end{abstract}
\maketitle

\section{Introduction}
\label{s:int}
Suppose that $S$ is the germ of an isolated normal complex surface
singularity. For hypersurface and complete intersection singularities, there
are natural smoothings (i.e., deformations with smooth generic fibre) given by
the defining functions, and their properties have been known for a long time:
such a smoothing is topologically a bouquet of 2--spheres. But in general it
is not clear whether smoothings of $S$ exist, or, if they do, what their basic
topological properties are. It would be natural to try to understand those
singularities which possess a smoothing with the `simplest' possible topology.
We say that a smoothing is a \emph{rational homology disk} ($\bfq$HD for
short) if the underlying smooth 4-manifold has rational homology groups
isomorphic to $H_* (D^4; \bfq )$, where $D^4$ denotes the 4-dimensional disk.
Strong constraints are imposed for a singularity to admit a $\bfq$HD smoothing
--- it is necessarily a \emph{rational} surface singularity, implying among
other things that the resolution graph of $S$ must be a (negative definite)
\emph{tree}, and the link of $S$ a rational homology sphere.  Examples of
singularities with $\bfq$HD smoothings already appeared in \cite{Wahl}. The
$\frac{p^2}{pq-1}$ cyclic quotient singularities ($0<q<p, (p,q)=1$) provide a
complete list of cyclic quotients with this property, and \cite{Wahl} also
contained some further examples (with resolution graphs given by
Figure~\ref{f:qhd3}(a)). In fact, throughout the years, a list of such
examples was compiled by J.~Wahl, which was known to the experts (cf.\ the
remark in \cite[bottom~of~page~505]{vSdJ}) but did not appear in print.

The smooth 4-manifold-theoretic application of certain singularities with
$\bfq$HD smoothings, through the rational blow-down procedure
(introduced by Fintushel and Stern \cite{FSratbl} and extended by Park
\cite{Pratbl}), have put the study of singularities with $\bfq$HD
smoothings at the forefront of 4-dimensional topology. In \cite{SSW}
a systematic investigation of the resolution graphs of such
singularities was initiated, and (relying on Donaldson's famous
Theorem A, and some further observations) strong combinatorial
constraints have been found for a (negative definite) plumbing tree to
be the resolution graph of a singularity admitting a $\bfq$HD
smoothing.  Although \cite{SSW} did not aim to provide a complete
classification of singularities with $\bfq$HD smoothings, the examples
given there (in hindsight) provided a nearly complete list of weighted
homogeneous singularities with $\bfq$HD smoothings (the only missing
examples from \cite{SSW} are the ones corresponding to the graphs of
Figures~\ref{f:qhd3}(h) and (i), which were also known to the authors of
\cite{SSW} to admit $\bfq$HD smoothings).

In the present work --- resting on results of \cite{SSW} and on some
fundamental theorems in symplectic geometry --- we give a complete
classification of the resolution graphs of weighted homogeneous
singularities admitting $\bfq$HD smoothings. Surprisingly enough, the
complete list of resolution graphs of weighted homogeneous
singularities with $\bfq$HD smoothings essentially coincides with the
list of examples of Wahl mentioned above. In order to state our
results precisely, we need a few preliminary notions and
definitions.

The link $Y_{\Gamma}$ of a singularity $S_{\Gamma}$ with resolution
graph $\Gamma$ is determined by the graph $\Gamma$, and according to
\cite{CNP} the 3-manifold $Y_{\Gamma}$ admits a (up to
contactomorphism) unique contact structure, its \emph{Milnor fillable
  contact structure} $\xi _{\G}$, given by the 2-plane field of complex
tangencies on $Y_{\Gamma}$ as a link of $S_{\Gamma}$.  Any smoothing
of the singularity $S_{\Gamma}$ naturally provides a Stein filling of
the Milnor fillable contact 3-manifold $(Y_{\Gamma}, \xi _{\G})$.
(For the definition of various notions of fillings of contact
3-manifolds, see \cite[Section~12.1]{OS}.)

\begin{defn}
{\rm We call a  normal complex surface singularity $S_{\Gamma}$
  \emph{spherical Seifert} if the link of the singularity is a Seifert
  fibred 3-manifold over the sphere $S^2$. The spherical Seifert
  singularity $S_{\Gamma}$ is \emph{small Seifert} if the link is a
  small Seifert fibred 3-manifold, i.e., it admits a Seifert fibration
  over $S^2$ with exactly three singular fibres.}
\end{defn}

A normal surface singularity is therefore spherical Seifert if and
only if it admits a resolution graph which is a star-shaped tree and
the vertices correspond to rational curves; in addition, $S_{\Gamma}$
is small Seifert if the central vertex (the unique vertex of valency
$>2$) in a minimal good resolution is of valency 3.  By
\cite[Theorem~2.6.1]{OW}, weighted homogeneous singularities with
rational homology sphere links are all spherical Seifert singularities
(but the converse does not hold). For a definition of weighted
homogeneous singularities (also called quasi--homogeneous, or
singularities with a good $\C ^*$--action) see, for example,
\cite[p.\ 206]{OW}.

\begin{defn}
{\rm Define ${\mathcal {QHD}}^3$ as the set of graphs given by
  Figure~\ref{f:qhd3}.
\begin{figure}[!ht]
\begin{center}
\includegraphics[width=13cm]{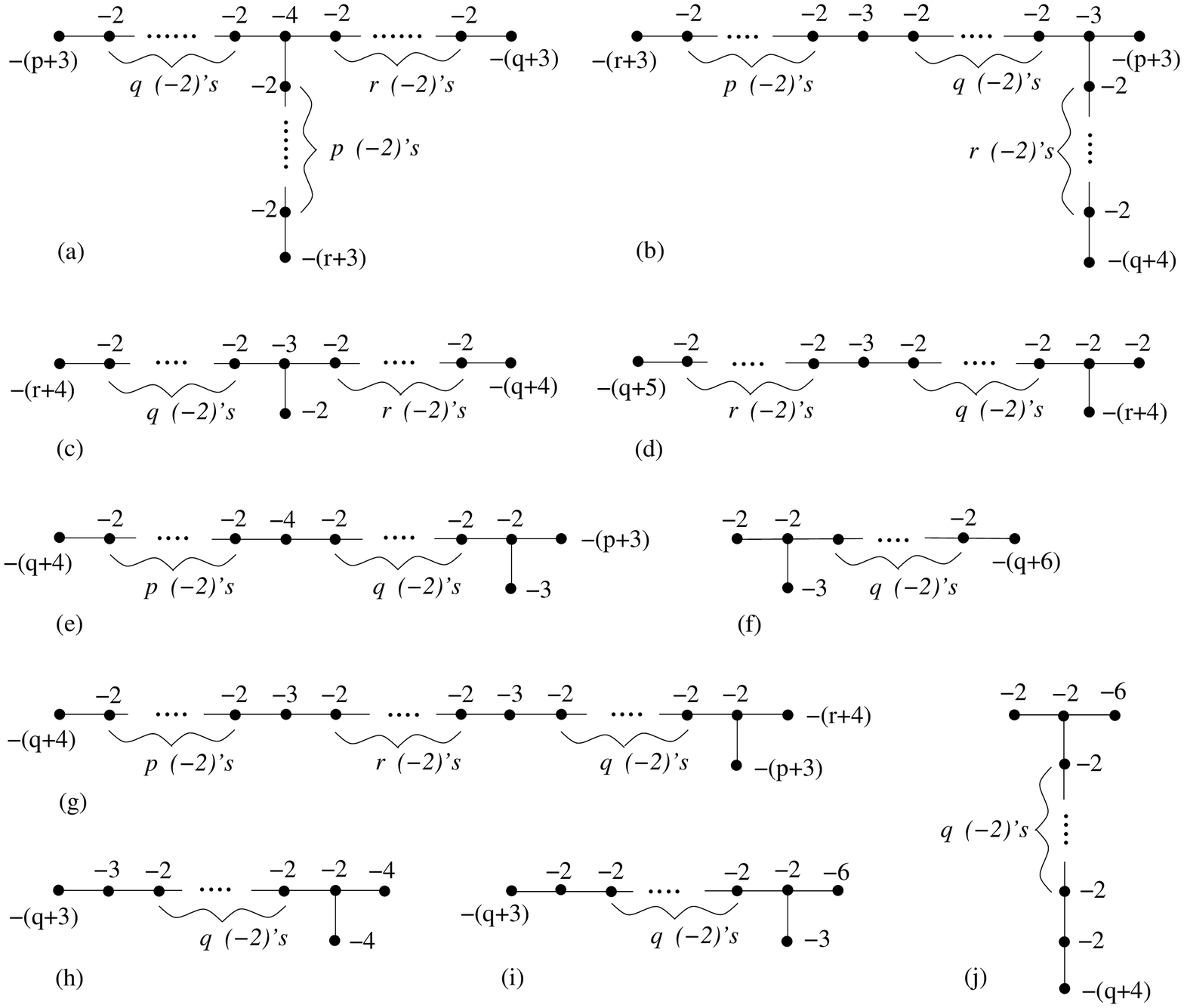}
\end{center}
\caption{The graphs defining the class ${\mathcal {QHD}}^3$ of
  plumbing graphs. We assume that $p,q,r\geq 0$.}
\label{f:qhd3}
\end{figure}
}
\end{defn}
\begin{rem}\label{r:wnm} {\rm The graphs given in Figure~\ref{f:qhd3}(a)
    form the set $\frw$ of \cite{SSW}, those in
    Figures~\ref{f:qhd3}(b) and (c) form $\frn$, while the collection
    of (d), (e), (f) and (g) were called $\farm$ in \cite{SSW}.  The
    graphs of Figure~\ref{f:qhd3}(h) provide the 1-parameter family
    ${\frb} ^3 _2$ of certain star-shaped graphs with three legs in
    the class $\frb$ of \cite{SSW}, and the ones of the form (i) and
    (j) are two 1-parameter families ${\frc }^3 _{2}$ and ${\frc }^3
    _{3}$ in $\frc$. (For the definition of the classes $\fra, \frb$
    and $\frc$ of graphs see Subsection~\ref{ss:fam}. The superscript
    in the notation is intended to indicate the number of legs; the
    subscripts in the cases of $\frb$ and $\frc$ will be explained in
    Subsections~\ref{ss:famc} and \ref{ss:famb}. With the same line of
    logic, families $\fra ^3, {\frb} ^3 _4$ and ${\frc}^3_6$ could also be
    defined, but these graphs already appear as (e) (with $p=0$), (d)
    (with $r=0$) and (f) of Figure~\ref{f:qhd3}.)}
\end{rem}
According to \cite{Lauf}, normal complex surface singularities
corresponding to the resolution trees in ${\mathcal {QHD}}^3$ are all
{\em taut}, that is, the resolution graph uniquely determines the
analytic structure of the corresponding singularity.  Since for any
star-shaped negative definite plumbing tree of spheres there is a
weighted homogeneous singularity with that resolution graph
\cite[Theorem~2.1]{Pinki}, the unique singularity above is necessarily
weighted homogeneous. The first main result of the paper is

\begin{thm}\label{t:main}
Suppose that $S_{\Gamma}$ is a small Seifert singularity with link
$Y_{\Gamma}$. Assume that $\Gamma $ is a minimal good resolution graph
of $S_{\Gamma}$, and therefore a negative definite star-shaped tree
with three branches.  Then the following three statements are
equivalent:
\begin{enumerate}
\item The singularity $S_{\Gamma}$ admits a $\bfq$HD smoothing.
\item The Milnor fillable contact structure on $Y_{\Gamma}$ admits a
weak symplectic $\bfq$HD filling.
\item The graph $\Gamma$ is in ${\mathcal {QHD}}^3$.
\end{enumerate}
\end{thm}
For star-shaped diagrams with more than three branches the analytic type of the
singularity is not determined by the graph itself, hence the
formulation of our result needs a little more care.
\begin{defn}
{\rm Define ${\mathcal {QHD}}^4$ as the union
of all graphs given by Figures~\ref{f:qhd4uj}(a), (b) and (c) for $n\geq 2$
in each case.
\begin{figure}[ht]
\begin{center}
\includegraphics[width=8cm]{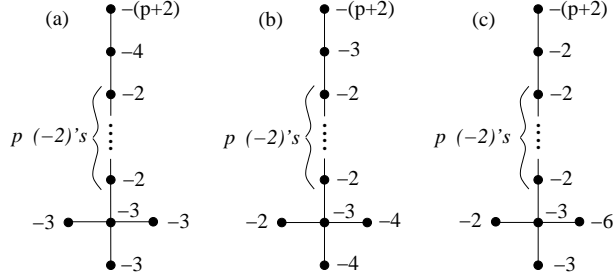}
\end{center}
\caption{The graphs of (a) define the class $\fra ^4$, graphs of (b)
  give the class $\frb ^4$, while the graphs of (c) give $\frc ^4$ (in
  all these cases we assume $p\geq 0$). The union of the above
  specified classes is, by definition, ${\mathcal {QHD}}^4$. Once
  again, the superscript in the notation records the number of legs of
  these star-shaped graphs.}
\label{f:qhd4uj}
\end{figure}
}
\end{defn}
According to \cite{Lauf}, the analytic type of a normal surface
singularity with resolution graph in ${\mathcal {QHD}}^4$ is
determined by the analytic type of the four intersection points of the
central curve $C$ with the branches, or equivalently, by the cross
ratio of these four points in $C$. In particular, all normal surface
singularities with these resolution graphs are weighted homogeneous.
With these remarks in place, we are ready to state the second main
result of the paper.

\begin{thm}\label{t:main2}
  Suppose that $\Gamma$ is a minimal, star-shaped plumbing tree with at least
  four branches, and the framing (i.e. weight) of the central vertex is less
  than $-2$. Then the following statements are equivalent.
\begin{enumerate}
\item There is a Seifert singularity $S_{\Gamma}$ with resolution
  graph $\Gamma$ which admits a $\bfq$HD smoothing.
\item The Milnor fillable contact structure on $Y_{\Gamma}$ admits a
weak symplectic $\bfq$HD filling.
\item The graph $\Gamma$ is in ${\mathcal {QHD}}^4$.
\end{enumerate}
\end{thm}

\begin{rems}
{\rm {\bf {(a)}} The assumption on the framing of the central vertex
  in Theorem~\ref{t:main2} is needed for our methods to work. In
  particular, a $(-2)$--framed central vertex with four legs provides
  a $(-2)$--curve in the dual configuration, hence the blow-down
  operation indicated by the dashed circles of Figures~\ref{f:c4leg},
  \ref{f:b4leg} and \ref{f:a4leg} cannot be started.  By accident,
  this assumption on the framing of the central vertex implies no
  constraint on the holomorphic result, since a normal surface
  singularity with $\bfq$HD smoothing is necessarily rational, hence
  the resolution graph does not admit a vertex for which the absolute
  value of the framing is strictly less than the valency of the vertex
  minus 1. The question of whether the Milnor fillable contact
  structure on the link of a normal surface singularity with
  star-shaped resolution tree, at least four branches and central
  framing $-2$ admits a weak symplectic $\bfq$HD filling is still
  open.

\noindent {\bf {(b)}} The above theorems concern exclusively the cases
when the resolution graph is star-shaped. No example of a normal
surface singularity with non-star-shaped minimal good resolution graph
which admits a $\bfq$HD smoothing is known. Partial results regarding
the nonexistence of $\bfq$HD smoothings follow from \cite{GaS, SSW,
  Wahluj}, but the lack of a convenient and general compactifying
divisor prevents us from treating the general case with methods
similar to the ones applied in the present paper.}
\end{rems}

The idea of the proof of the main results can be summarized as
follows. First of all, the implication $(1)\Rightarrow (2)$ in both
theorems follows from the general principle that any smoothing of a
singularity is a weak symplectic filling of the Milnor fillable
contact structure on the link of the singularity. The implication
$(3)\Rightarrow (1)$ (which was mostly already verified in
\cite{SSW}) in both statements requires the construction of $\bfq$HD
fillings; in the cases not covered by \cite{SSW} we will apply the
method of smoothings of negative weight. In order to prove
$(2)\Rightarrow (3)$ we need to show that for any star-shaped
resolution graph outside ${\mathcal {QHD}}^3$ and ${\mathcal {QHD}}^4$
the Milnor fillable contact structure admits no symplectic $\bfq$HD
filling.  These nonexistence results rely on deep symplectic geometric
theorems (most importantly on McDuff's result regarding symplectic
manifolds containing symplectic spheres of self-intersection number $1$)
and tedious combinatorial arguments. In principle these arguments
could be extended to classify other types of symplectic fillings, but
the combinatorics (which is already quite delicate for the case of
$\bfq$HD fillings) can become extremely complex to handle.

Finally a few words about the use of symplectic geometry. In order to
show that certain singularities \emph{do not} admit $\bfq$HD
smoothings, we will apply the following strategy: first we will
construct a fixed symplectic manifold for the singularity at hand
(which we will call the \emph{compactifying divisor}) and glue the
hypothesized $\bfq$HD weak symplectic filling to it in a symplectic
manner.  In the resulting \emph{closed} symplectic manifold we then
locate a curve configuration, which will lead to some geometric
contradiction unless the singularity had resolution tree from
${\mathcal {QHD}}^3$ or ${\mathcal {QHD}} ^4$.  Although both the
compactifying divisor and the hypothesized smoothing are holomorphic
objects, we do not know any holomorphic way to glue them together to
obtain a globally holomorphic closed manifold, on which then
algebro-geometric methods would be applicable. An alternative,
algebraic geometric compactification of the smoothings can be achieved
by applying the method of deformations of `weight less than or equal
to zero'. As we were informed by J. Wahl \cite{personal}, the
necessary results can be proved using delicate methods of complex
algebraic geometry and singularity theory.  From that point on, the
adaptation of our combinatorial arguments follow in a fairly
straightforward manner. We decided to use the symplectic geometric
methods, since in this way the resulting theorem becomes stronger in
the aspect of getting obstructions even for the existence of $\bfq$HD
weak fillings. Also, by completing the arguments in the symplectic
setting, our result shows yet another instance where objects behave in
a parallel manner in the complex analytic and in the symplectic
category.

The paper is organized as follows. In Section~\ref{s:prelim} the
symplectic geometric preliminaries used in the proofs of the main
results are listed, together with a quick outline of the ideas
employed in the later arguments.  Section~\ref{s:threeleg} deals with
small Seifert singularities, i.e., with those singularities which have
star-shaped minimal good resolution graphs with three branches.
Finally, in Section~\ref{s:fourleg}, we address the general case of
spherical Seifert singularities.

{\bf Acknowledgements:} AS was supported by the Clay Mathematics
Institute and by OTKA T67928. Both authors acknowledge support by
Marie Curie TOK project BudAlgGeo. We would like to thank Ron Stern
for many useful correspondences and Jonathan Wahl for helping us with
the smoothing theory of normal surface singularities and suggesting important
improvements and corrections in the text. Finally, we would like to
thank the anonymous referee for useful comments and corrections.

\section{Preliminaries}
\label{s:prelim}
\subsection{Symplectic geometric preliminaries}
Our results rely on the following fundamental theorem due to McDuff.

\begin{thm}(McDuff, \cite[Theorem~1.4]{M})
\label{thm:McDuff}
Let $(M, \omega)$ be a closed symplectic $4$-manifold. If $M$ contains
a symplectically embedded $2$-sphere $L$ of self-intersection number
$1$, then $M$ is a rational symplectic $4$-manifold. In particular,
$M$ becomes a the complex projective plane after blowing down a finite
collection of symplectic $(-1)$-curves away from $L$. \qed
\end{thm}

The following two lemmas are based on the above theorem of McDuff and the
details of the proofs can be found in \cite{BO}:

\begin{lem} (Cf.\ \cite[Lemma 2.13]{BO})
\label{lem:exccurves}
Let $(M, \omega)$ be a closed symplectic $4$-manifold containing a
symplectically embedded $2$-sphere $L$ of self-intersection number $1$
and a collection of symplectically immersed $2$-spheres $C_1, \dots,
C_k$. Suppose that $J$ is a tame almost complex structure for which
$L$, $C_1, \dots ,C_k$ are pseudoholomorphic. Then there exists at
least one $J$-holomorphic $(-1)$-curve in $M \setminus L$ unless
$L\cdot C_i>0$ and $C_i\cdot C_i = (L\cdot C_i)^2$ for all $i$. \qed
\end{lem}

\begin{lem} (\cite[Lemma 2.5]{BO})
\label{lem:nonex}
 Let $M$ be a closed symplectic $4$-manifold containing a
 symplectically embedded $2$-sphere $L$ of self-intersection number
 $1$.  If $C$ is an irreducible singular or higher genus
 pseudoholomorphic curve in $M$, then $C\cdot L\geq 3$. In particular
 there are no irreducible singular or higher genus pseudoholomorphic
 curves in $M \setminus L$. \qed
\end{lem}
This lemma has the following simple corollary.
\begin{cor}\label{c:nocycle}
 Let $M$ be a closed symplectic $4$-manifold containing a
 symplectically embedded $2$-sphere $L$ of self-intersection number
 $1$. Then there is no cycle of pseudoholomorphic spheres in the
 complement $L$.
\end{cor}
\begin{proof}
If such a cycle existed, by gluing adjacent components around the
nodes we would be able to construct an embedded pseudoholomorphic
curve of genus $1$ which would contradict Lemma~\ref{lem:nonex}.
\end{proof}
Another fact which we will frequently use is that for any almost
complex structure $J$ on a 4-manifold $X$ any intersection point of
two $J$-holomorphic curves $C_1$ and $C_2$ contributes positively to
the algebraic intersection number $C_1\cdot C_2$.

The next lemma
easily follows from McDuff's Theorem~\ref{thm:McDuff}.
\begin{lem} \label{lem:nonexnonneg}
Let $M$ be a closed symplectic $4$-manifold containing a
symplectically embedded $2$-sphere $L$ of self-intersection number
$1$. Then there is no symplectically embedded sphere of nonnegative
self intersection number in the complement of $L$.
\end{lem}
\begin{proof}
Since $M$ is rational, it follows that $b_2^+(M)=1$, immediately
implying the lemma. (Notice that a symplectic sphere of any
self-intersection --- including $0$ --- is homologically essential.)
\end{proof}
\begin{lem} \label{l:sing}
Suppose that $C\subset \cpk$ is a $J$-holomorphic curve for some tame
almost complex structure $J$, in the homology class $[C]=d[\cp ^1]$,
and $C$ has at least two singular points. Then $d\geq 4$.
\end{lem}
\begin{proof}
The $J$-holomorphic line passing through two singular points
intersects $C$ with multiplicity at least 4, providing the result.
\end{proof}
We record here the following fact which we will apply repeatedly in
the sequel: By the adjunction formula, a pseudoholomorphic rational
curve representing the class $3 [\cp ^1]$ in $\cpk$ must be either
immersed with exactly one node (that is a point where two branches of
the curve intersect transversely) or it must have exactly one
nonimmersed point which is necessarily a $(2,3)$-cusp
singularity. (Here a pseudoholomorphic curve in a $4$-manifold is said
to have a $(2,3)$-cusp singularity if there is a parametrization
around the singular point in which the curve has the form
$(z^2,z^3)+O(4)$, see \cite{M2}.) In conclusion, the link of such a curve
around its singular point is either connected (and is the trefoil knot)
or has two components (and is the Hopf link).

\subsection{The families $\fra, \frb$ and $\frc$}
\label{ss:fam}
The three inductively defined families $\fra , \frb, \frc$ of graphs
found in \cite{SSW} will play a central role in our subsequent
arguments.  For the sake of completeness, we shortly recall the definition of
these families below.

Let us define $\fra$ as the family of graphs we get in the following
way: start with the graph of Figure~\ref{f:alap}(a), blow up its
$(-1)$-vertex or any edge emanating from the $(-1)$-vertex and
repeat this procedure of blowing up (either the new $(-1)$-vertex or
an edge emanating from it) finitely many times, and finally modify the
single $(-1)$-decoration to $(-4)$. Depending on the number and
configuration of the chosen blow-ups, this procedure defines an
infinite family of graphs. Define $\frb$ similarly, this time starting with
Figure~\ref{f:alap}(b) and substituting $(-1)$ in the last step with
$(-3)$, and finally define $\frc$ in the same vein by starting with
Figure~\ref{f:alap}(c) and putting $(-2)$ in the place of $(-1)$ in the
final step.
\begin{figure}[ht]
\begin{center}
\setlength{\unitlength}{1mm}
\unitlength=0.5cm
\begin{graph}(14,4.5)(0,-3)
\graphnodesize{0.2}

 \roundnode{m1}(-2,0)
 \roundnode{m2}(0,0)
 \roundnode{m3}(2,0)
 \roundnode{m4}(0,-2)

 \roundnode{m5}(12,0)
 \roundnode{m6}(14,0)
 \roundnode{m7}(16,0)
 \roundnode{m8}(14,-2)

 \roundnode{m9}(5,0)
 \roundnode{m10}(7,0)
 \roundnode{m11}(9,0)
 \roundnode{m12}(7,-2)

\edge{m1}{m2}
\edge{m2}{m3}
\edge{m2}{m4}

\edge{m5}{m6}
\edge{m6}{m7}
\edge{m6}{m8}

\edge{m9}{m10}
\edge{m10}{m11}
\edge{m10}{m12}

\freetext(0,-3){(a)}
\freetext(14,-3){(c)}
\freetext(7,-3){(b)}

  \autonodetext{m1}[n]{{\small $-3$}}
  \autonodetext{m2}[n]{{\small $-1$}}
  \autonodetext{m3}[n]{{\small $-3$}}
  \autonodetext{m4}[e]{{\small $-3$}}

  \autonodetext{m5}[n]{{\small $-2$}}
  \autonodetext{m6}[n]{{\small $-1$}}
  \autonodetext{m7}[n]{{\small $-6$}}
  \autonodetext{m8}[e]{{\small $-3$}}

  \autonodetext{m9}[n]{{\small $-2$}}
  \autonodetext{m10}[n]{{\small $-1$}}
  \autonodetext{m11}[n]{{\small $-4$}}
  \autonodetext{m12}[e]{{\small $-4$}}
\end{graph}
\end{center}
\caption{\quad Nonminimal plumbing trees giving rise to the families $\fra,
  \frb$ and $\frc$.}
\label{f:alap}
\end{figure}

The starting point of the proofs of Theorems~\ref{t:main} and
\ref{t:main2} rests on the main result of \cite{SSW} which can be
summarized as follows.  Recall the definitions of $\frw, \farm, \frn$
from Remark~\ref{r:wnm} and let $\frg$ denote the set of plumbing
chains with framings determined by the negatives of the continued
fraction coefficients of the rational numbers of the form
$\frac{p^2}{pq-1}$ for all $0<q<p$ and $(p,q)=1$.

\begin{thm}(\cite{SSW})\label{t:ssw}
Suppose that $\Gamma$ is a minimal, negative definite plumbing
tree. If it gives rise to a surface singularity $S_{\Gamma}$ admitting
a $\bfq$HD smoothing, or if the Milnor fillable contact structure on
the corresponding plumbing 3-manifold $Y_{\Gamma}$ admits a $\bfq$HD
filling then $\Gamma$ is in $\frg \cup \frw \cup \frn \cup \farm \cup
\fra \cup \frb \cup \frc$. \qed
\end{thm}

\subsection{Outline of the proof of $(2)\Rightarrow (3)$ in the main theorems}

Suppose that $\G$ is a graph of the type considered in
Theorems~\ref{t:main} or~\ref{t:main2}.  Let $Y_\G$ denote the
associated plumbed $3$-manifold and $\xi_\G$ the unique Milnor
fillable contact structure on $Y_\G$. According to
Theorem~\ref{t:ssw}, if $(Y_\G,\xi_\G)$ admits a symplectic $\bfq$HD
filling then $\G$ must be in $\frw\cup \frn \cup \farm\cup \fra \cup
\frb \cup \frc$. Since by \cite[Section~8]{SSW} the singularities
corresponding to graphs in $\frw\cup \frn \cup \farm$ admit $\bfq$HD
smoothings, the corresponding links admit symplectic $\bfq $HD
fillings.  Therefore we only need to consider star-shaped graphs in
$\fra \cup \frb \cup \frc$; let $\G$ be such a graph with $s$ legs
$\ell _1, \ldots , \ell _s$ and with central framing $-b$. Suppose that
the framing coefficients along the leg $\ell _i$ are given by the
negatives of the continued fraction coefficients of
$\frac{n_i}{m_i}>1$. Consider then the ``dual'' graph $\G '$ which is
star-shaped with $s$ legs $\ell _1', \ldots , \ell _s'$, central
framing $b-s$, and framings along leg $\ell _i '$ given
by the negatives of the continued fraction coefficients of
$\frac{n_i}{n_i-m_i}$. Let $W_\G$ and $W_{\G '}$ denote the
corresponding plumbing 4-manifolds.

\begin{lem}[Cf., for example, \cite{SSW}]
Suppose that $\G, \G'$ are star-shaped plumbing trees as above.  The
boundary of $W_\G$ is orientation preserving diffeomorphic to the link
$Y_\G$, while $\partial W_{\G '}=-Y_{\G}$. In addition, $W_\G \cup
W_{\G'}$ is a 4-manifold diffeomorphic to $\cpk \# m \cpkk$ for some
positive integer $m$.
\end{lem}
\begin{proof}[Proof (sketch)]
Consider the Hirzebruch surface ${\mathbb {F}}_{b}$ with zero-section
of self-intersection $-b$ (and hence with infinity-section of
self-intersection $b$). Fix $s$ distinct fibres of the $\cp
^1$-fibration and blow up the intersection points of these fibres with
the infinity-section. After the appropriate sequence of blow-ups we
can identify in the resulting rational surface a configuration of
curves intersecting each other according to $\G$, and it is easy to
see that the complementary curves will intersect each other according
to $\G '$.  Since the curves intersecting according to the graph $\G$
admit an $\omega$-convex neighbourhood (with the symplectic
form $\omega$ being the K\"ahler form on the Hirzebruch surface), with
the Milnor fillable contact structure as induced structure on the
boundary, the complement (diffeomorphic to $W_{\G '}$) provides a
strong concave filling of $(Y_\G , \xi _\G )$.  Since the complement
is also a regular neighbourhood of a configuration $K$ of curves
(intersecting each other according to $\G '$), we will refer to $K$
(and sometimes, with a slight abuse of notation, to the regular
neighbourhood $W_{\G '}$) as the \emph{compactifying divisor}.
\end{proof}
Suppose now that $X$ is a weak symplectic $\bfq$HD
filling of $(Y_\G,\xi_\G)$. Since $Y_\G$ is a rational homology
$3$-sphere, we can perturb the symplectic structure on $X$ in a
neighbourhood of the boundary so that it becomes a strong symplectic
filling of $(Y_\G,\xi_\G)$. Glue $X$ and $W_{\G '}$ along $Y_\G$ to
obtain a closed symplectic $4$-manifold $Z$.
Let $k$ denote the number of irreducible components of the compactifying
divisor $K$. Then since $W_{\G'}$ is a regular neighbourhood of $K$, we have
that $b_2(W_{\G'})=k$. Since $X$ is a $\bfq$HD, it follows that $b_2(Z)=k$.

In all cases that we consider, it turns out that $K$ (after, possibly,
some blow-downs) contains a component which is a sphere that is
embedded in $W_{\G'}\subset Z$ with self intersection number
$1$. (This is the step when the assumption $\G \in \fra \cup \frb
\cup \frc$, and the constraint of Theorem~\ref{t:main2} on the framing
of the central vertex become crucial.) Let $L$ denote one such
component. By McDuff's Theorem~\ref{thm:McDuff}, we conclude that $Z$
is a rational symplectic $4$-manifold and hence diffeomorphic to $\cpk
\# (k-1)\cpkk$. In fact, McDuff's Theorem implies that for a generic
tame almost complex structure $J$, in the complement of $L$ we can
find $k-1$ disjoint embedded symplectic $2$-spheres with
self-intersection number $-1$ (we will refer to these as {\em
  symplectic $(-1)$-curves}), and after blowing these down we obtain
$\cpk$. However, we would like to understand how the other components
of $K$ descend under the blowing down map. We thus proceed as follows.

We choose a tame almost complex structure $J$ on $Z$ with respect to
which all the curves in $K$ are pseudoholomorphic. We assume that $J$
is generic among those almost complex structures for which $K$ is
$J$-holomorphic.  Appealing to Lemma \ref{lem:exccurves} we can find a
pseudoholomorphic $(-1)$-curve $E$ in $Z$ disjoint from $L$. By
perturbing the almost complex structure $J$ if necessary, we can
assume that $E$ intersects each component of $K$ transversely and does
not pass through any point where two or more components of $K$
intersect. We choose a maximal family $\{E_j\}$ of such
pseudoholomorphic $(-1)$-curves which are disjoint from $L$ and blow
them down. Let $Z^\prime$ denote the resulting symplectic
$4$-manifold.

By \cite[Lemma~4.1]{OO}, we can find a tame almost complex structure
$J^\prime$ on $Z^\prime$ with respect to which the images of all the
components of $K$ are pseudoholomorphic. We will again be in the
situation where we have a closed symplectic $4$-manifold containing a
symplectically embedded $2$-sphere of self-intersection number $1$ and
a collection of symplectically immersed $2$-spheres (the images of the
components of $K-L$). Let $K^\prime$ denote the image of $K$ under the
blowing down map. If $K^\prime$ contains a curve disjoint from $L$ (as
will always be the case in the situations we consider), then we can
again appeal to Lemma \ref{lem:exccurves} and find a pseudoholomorphic
$(-1)$-curve $E^\prime$ in $Z^\prime\setminus L$.

Note that $E^\prime$ must be a component of $K^\prime$. Indeed, assume
to the contrary that $E^\prime$ is not a component of
$K^\prime$. Perturbing the almost complex structure slightly, we may
assume that $E^\prime$ does not pass through the images of the
blown-down $(-1)$-curves $E_j$. Hence we may assume that $E^\prime$ is
actually a pseudoholomorphic $(-1)$-curve already in $Z\setminus L$,
which contradicts the maximality of $\{E_j\}$.

By suitably perturbing the almost complex structure, we can arrange
that $E^\prime$ intersects each component of $K^\prime- E^\prime$
transversely and it does not pass through any point where two or more
components of $K^\prime- E^\prime$ meet. We then blow down
$E^\prime$. Let $Z^{''}$ denote the resulting ambient symplectic
$4$-manifold and $K^{''}$ denote the image of $K^\prime$.

As before, we can again check that there are no pseudoholomorphic
$(-1)$-curves in $Z^{''}$ except possibly for some components of
$K^{''}$. Perturbing the almost complex structure as before, blowing
down these pseudoholomorphic $(-1)$-curves and proceeding in this way,
we must eventually obtain $\cpk$ together with a symplectically
embedded $2$-sphere of self-intersection number $1$ and a collection
of symplectically immersed $2$-spheres. Since we are assuming that $X$
is a $\bfq$HD, it follows that we must obtain $\cpk$ after $k-1$ blow
downs and the configuration $K$ must descend to a valid configuration
in $\cpk$.  This places strong restrictions on the combinatorial
structure of $K$: all components of $K$ which are disjoint from $L$
must be blown down at some point of this procedure (so in particular
they must become $(-1)$-curves at some earlier point), while a
component $K_0$ of $K$ intersecting $L$ must become a $J$-holomorphic
submanifold of $\cpk$ of degree $K_0\cdot L$.  This condition, for
example, determines the homological square of the image of $K_0$ in
$\cpk$, and for low degrees it also determines the topology of the
result. For most graphs $\G$ we will reach a homological contradiction
at some point of this procedure, showing the nonexistence of the
hypothesized $\bfq$HD filling $X$.

\section{Small Seifert singularities}
\label{s:threeleg}
By Theorem~\ref{t:ssw} and by the fact that all graphs in $\frw \cup
\frn \cup \farm$ are known to admit $\bfq$HD smoothings \cite{SSW},
we only need to examine the three-legged graphs in $\fra\cup \frb \cup
\frc$. The discussion will be given for each of these classes
separately; for technical reasons we start with the case of graphs in
$\frc$.

\subsection{Graphs in $\frc$}
\label{ss:famc}
Recall that graphs in $\frc$ are defined by repeatedly blowing up the
basic configuration shown by Figure~\ref{f:alap}(c) and then replacing
the $(-1)$-framing with $(-2)$. To get three-legged graphs, we only
blow up \emph{edges} emanating from the $(-1)$-vertex. There are three
cases we distinguish depending on which edge we blow up in the first
step in the basic example. The index of the subfamily records the
(negative of the) framing of the leaf to which the first blown up edge
points. Notice that the families $\frc ^3_2$ and $\frc _3^3$ defined
by the graphs of (i) and (j) of Figure~\ref{f:qhd3} are subfamilies of
$\frc _2$ and $\frc _3$, respectively.

\smallskip
\noindent {\bf{The family $\frc _6$:}} Consider the generic member of
the family $\frc _6$ depicted in Figure~\ref{f:c6}(a).
\begin{figure}[ht]
\begin{center}
\includegraphics[width=10cm]{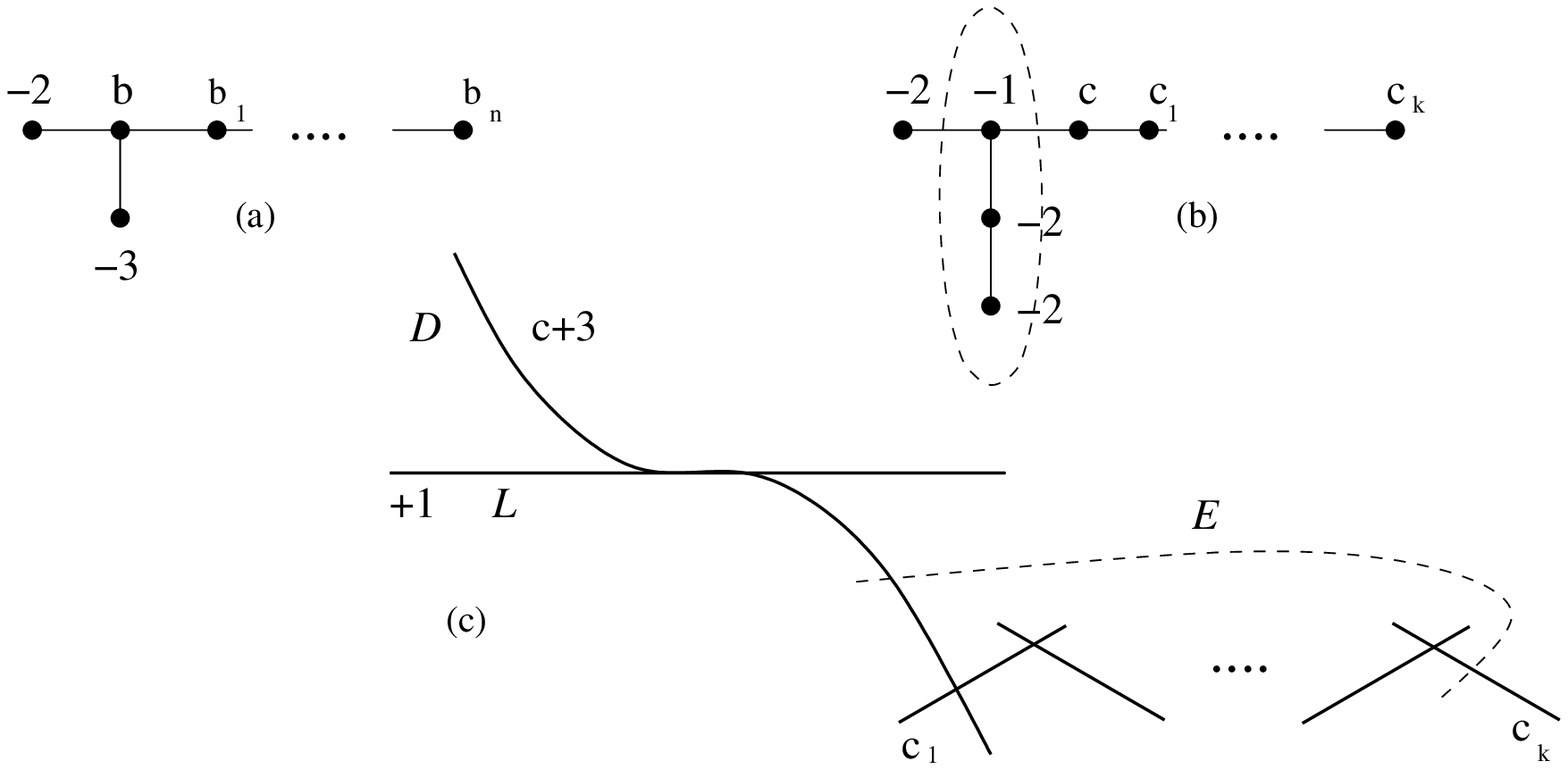}
\end{center}
\caption{The generic graph, its dual, and the configuration of curves
  after 3 blow-downs in the family $\frc _6$.}
\label{f:c6}
\end{figure}
The dual graph (after possibly repeatedly blowing up the edge
emanating from the central vertex towards the long leg until the
central framing becomes $-1$) has the shape given by
Figure~\ref{f:c6}(b). Blowing down the central vertex together with
the two $(-2)$'s (encircled by the dashed circle in
Figure~\ref{f:c6}(b)), we arrive at the diagram of
Figure~\ref{f:c6}(c); here the curves are symbolized by arcs, and the
intersection of two arcs means that the two corresponding curves
intersect each other. (The dashed arc of Figure~\ref{f:c6}(c) will be
relevant only at some later point of the argument.) The resulting
$(+1)$-curve will be denoted by $L$, while the curves of the long leg
(with framings $c,c_1, \ldots , c_k$) will become $D, C_1, \ldots ,
C_k$, respectively.  The tangency between $D$ and $L$ is a triple
tangency. (We use a straight line to indicate $L$ and a cubic curve to
picture $D$, which eventually will become a singular cubic in $\cpk$.)
Since $b_n\leq -6$, it is easy to see that $k\geq 3$.  Notice also
that $c_i\leq -2$ once $i\geq 1$ and $c$ is negative. By gluing this
compactifying divisor to a potentially existing $\bfq$HD filling $X$
we get a closed symplectic manifold $Z$ with $b_2(Z)=k+2$. The
symplectic 4-manifold $Z$ obviously contains a symplectic
$(+1)$-sphere (namely, the curve $L$), hence it follows by McDuff's
Theorem~\ref{thm:McDuff} that $Z$ is a rational symplectic
$4$-manifold, that is, a symplectic blow-up of $\cpk$ at a finite
number of points, hence $Z$ is diffeomorphic to $\cpk \# (k+1)
\overline{\cpk}$. By repeated applications of
Lemma~\ref{lem:exccurves}, we can blow down the pair $(Z,L)$ to obtain
$(\cpk , {\rm line})$, while preserving the pseudoholomorphicity of
the images of $D,C_1,\ldots,C_k$. Since the curves $C_1,\ldots,C_k$ in
the chain are disjoint from the $(+1)$-curve $L$ and are homologically
essential, we must blow them down, while the curve $D$ will descend to
a cubic curve in $\cpk$. Since the resulting cubic curve will be the
image of a rational curve, it necessarily must contain a singular
point. The above observations imply, therefore, that there is a unique
additional $(-1)$-curve $E$ in $Z$ for the chosen almost complex
structure, which we have to locate in the diagram.  Since
$J$-holomorphic curves intersect positively, the geometric
intersections in these cases can be computed via homological
arguments.
\begin{prop}\label{p:c6}
Under the above circumstances the exceptional divisor $E$ must
intersect the curve $D$ and the curve $C_k$ in the chain in one point
each. Consequently, the framings should satisfy $c_i=-2$ for $i=1,
\ldots , k$ and $c=-k+2$. In particular, the resolution graph of the
singularity (given by Figure~\ref{f:c6}(a)) must be of the form given
in Figure~\ref{f:qhd3}(f).
\end{prop}
\begin{proof}
Let $\J_K$ denote the nonempty set of tame almost complex structures
on $Z$ with respect to which all the curves of $K=L\cup D\cup C_1 \cup
\ldots \cup C_k$ are pseudoholomorphic. Choose an almost complex
structure $J$ which is generic in $\J_K$. If we blow down all
$J$-holomorphic $(-1)$-curves away from $L$, we can show that the
chain $C_1,\ldots,C_k$ is transformed into a configuration of curves
which can be sequentially blown down. There must be precisely one
$(-1)$-curve $E$ in the complement of $L$ which is not contained in
the chain $C_1,\ldots,C_k$; this $(-1)$-curve $E$ must intersect the
chain to start its sequential blow-down. $E$ also must intersect the
curve $D$ at least once, since (as $D$ has intersection number $3$
with the $(+1)$-curve $L$) $D$ will become a singular cubic curve in
$\cpk$. By Corollary~\ref{c:nocycle} the curve $E$ cannot intersect
the long chain twice. With a similar argument we can see that it can
intersect the chain only in its endpoints: if it intersects the chain
in a curve $C_i$ which is not at one of its ends, then blowing down
$E$ we get a curve $C_i^\prime$ which now intersects $D$ and two
further curves in the chain.  When we blow down $C_i^\prime$, the two
neighbours will pass through the same point of $D$. If, now, the image
of $C_{i-1}$ is the next curve of the chain to get blown down, then
the images of all curves in the portion $C_1,\ldots,C_{i-1}$ of the
chain must get blown down before the image of the curve $C_{i+1}$ is
blown down. Otherwise, we will get a singular point on the image of
$D$ and at least one further curve of the chain passing through
through that singular point. After a slight perturbation of the almost
complex structure, when (the image) of one of these curves is
eventually blown down we will get a further singular point on the
image of $D$, which (with the aid of Lemma~\ref{l:sing}) provides a
contradiction. However, after the images of $C_1,\ldots,C_{k-1}$ are
blown down, the image of $D$ will become singular, and the same
argument again provides a contradiction. If the image of $C_{i+1}$ is
the next curve of the chain to get blown down after $C_i '$, then, as
before, we can argue that the images of all curves in the portion
$C_{i+1},\ldots,C_{k}$ of the chain must get blown down before the
image of the curve $C_{i-1}$ is blown down. If $i>3$, then, when the
image of $C_{i-1}$ is blown down, we will get a contradiction as
before. If $i=3$, then, when image of $C_{i-1}$ is blown down, we will
obtain a singular point on the image of $D$ which has multiplicity
greater than $2$ and hence its link will not be the trefoil knot or
the Hopf link, a contradiction.

If $E$ intersects the chain on its end near $D$, then after the second
blow-down $D$ develops a transverse double point singularity, and the
further blow-downs then create more singular points (in the spirit of
the argument above), leading to a curve which cannot represent three
times the generator in the complex projective plane. Hence the only
possibility for the $(-1)$-curve $E$ is to intersect the chain at its
farther end, and intersect $D$ once (as shown by the dashed curve $E$
of Figure~\ref{f:c6}(c)). In order to blow down all the curves in the
chain we must have $c_i=-2$ for $i=1,\ldots , k$, and since the
self-intersection of $D$ will become 9 after all the blow-downs, we
derive $c=-k+2$. With this last observation, and a simple computation
of the dual graph, the proof is complete.
\end{proof}

\smallskip
\noindent {\bf{The family $\frc _3$:}} The generic member of this family is
given by Figure~\ref{f:c3}(a), together with the dual graph and the
result of the triple blow-down.  (Once again, we disregard the dashed
arcs of Figure~\ref{f:c3}(c) momentarily.)
\begin{figure}[ht]
\begin{center}
\includegraphics[width=10cm]{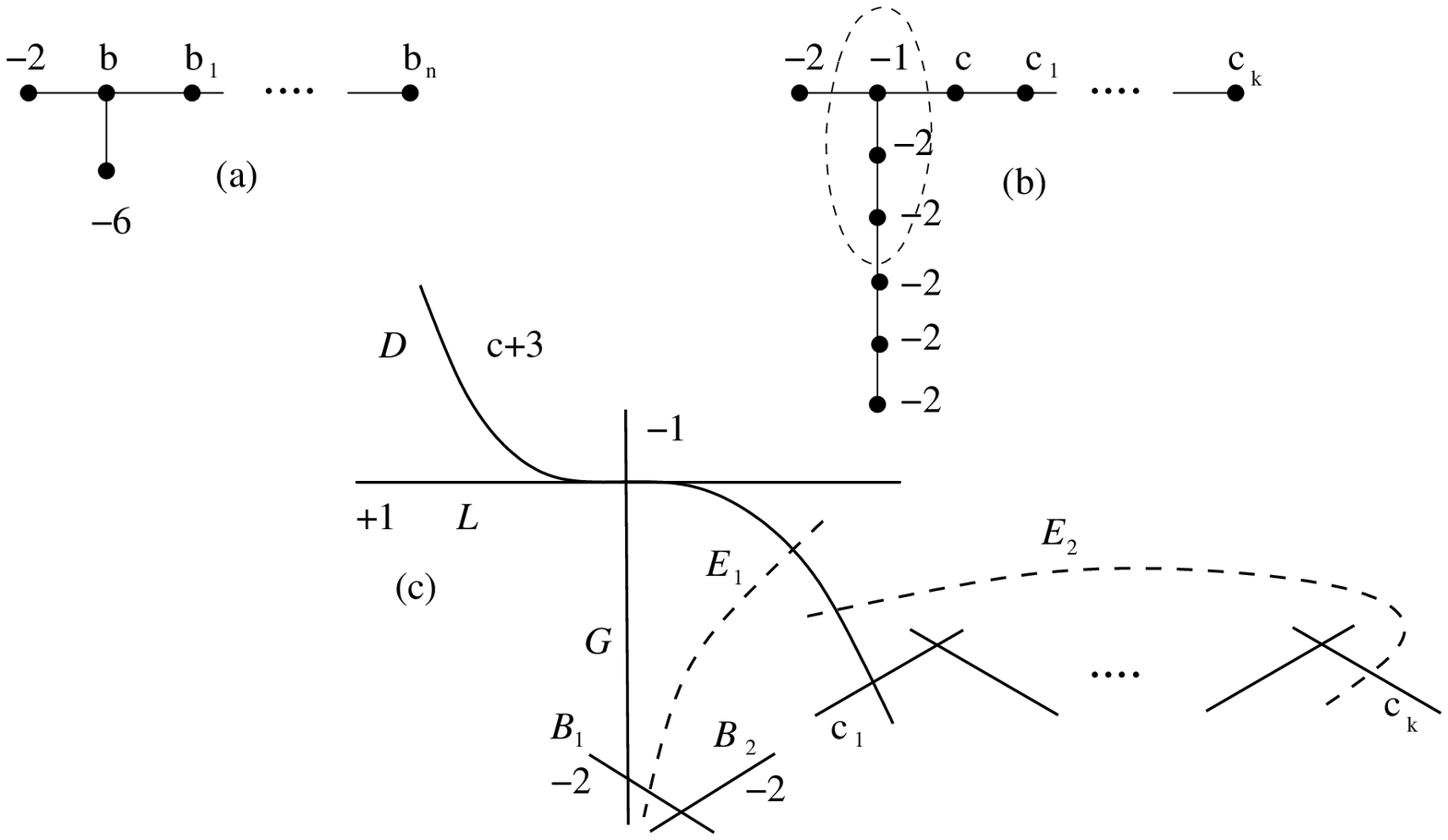}
\end{center}
\caption{The generic graph, its dual, and the configuration of curves
  after 3 blow-downs in the family $\frc _3$. The curves $E_1, E_2$
  are only shown for the first possibility given by
  Proposition~\ref{p:c3pr}.}
\label{f:c3}
\end{figure}
By gluing the compactifying divisor given by Figure~\ref{f:c3}(c) to a
potentially existing $\bfq$HD filling $X$ we get a closed symplectic
manifold $Z$, and a simple count shows that $b_2(Z)=k+5$. The
symplectic 4-manifold $Z$ obviously contains a symplectic
$(+1)$-sphere (namely, the curve $L$), hence, by McDuff's
Theorem~\ref{thm:McDuff}, $Z$ is diffeomorphic to $\cpk \# (k+4)
\overline{\cpk}$. By repeated applications of
Lemma~\ref{lem:exccurves}, we can blow down the pair $(Z,L)$ to obtain
$(\cpk, {\rm line})$, while preserving the pseudoholomorphicity of the
images of $D,C_1,\ldots,C_k,B_1,B_2$. Since the curves $C_1, \ldots ,
C_k$ and $B_1,B_2$ are disjoint from the $(+1)$-curve $L$ and are
homologically essential, we must blow them down. This means that there
are two further $(-1)$-curves $E_1$ and $E_2$ which we have to locate
in the diagram.  For a generic almost complex structure these curves
will be $(-1)$-curves disjoint from each other. Since both $B_1$ and
$B_2$ have to be blown down (being disjoint from the $(+1)$-curve
$L$), one of them must intersect one of the $(-1)$-curves, say $E_1$.
Since the complement of the $(+1)$-curve does not contain
homologically essential spheres with nonnegative square, $E_2$ then
cannot intersect any of the $B_i$.
\begin{prop}\label{p:c3pr}
Under the above circumstances, the existence of a $\bfq$HD smoothing
$X$ implies that $E_2$ intersects $D$ and $C_k$, and $E_1$ either
intersects $B_1$ and $D$ or $B_2$ and $C_1$. The self-intersections in
these two cases are $c=-k-1$ and $c_1=\ldots =c_k=-2$ or $c=-k+2$,
$c_1=-5$ and $c_2=\ldots =c_k=-2$. In particular, the resolution graph
in the first case is given by Figure~\ref{f:qhd3}(j), while in the
second case by Figure~\ref{f:qhd3}(d) (with $q=k-4$ and $r=2$).
\end{prop}
\begin{proof}
{\bf Case I:} Suppose that $E_1\cdot B_1>0$. After three blow-downs
the curve $G$ becomes a $(+1)$-curve, so it cannot be blown down any
further: in $\cpk$ it will be a curve intersecting the $(+1)$-curve
once, hence it will be a line with self-intersection number $1$. Therefore,
to prevent further blow-downs along the points of the vertical curve,
$E_2\cdot G=0$ and $E_1$ must be disjoint from the long chain. So
$E_2$ must intersect the long chain, and since the whole chain must be
blown down, a simple adaptation of the proof of Proposition~\ref{p:c6}
gives that the only possibility for $E_2$ is the one described in the
statement. Notice that the images of $G$ and $D$ must intersect each
other three times after all curves have been blown down, which can be
achieved only if $E_1$ intersects $D$ exactly once. (Recall that $E_2$
must stay disjoint from $G$.)  This argument shows that the only
possibility for $E_1$ and $E_2$ (under the assumption $E_1\cdot B_1 >
0$) is given by the dashed lines of Figure~\ref{f:c3}(c), providing
the first set of values of $c$ and $c_i$.

{\bf Case II:} Suppose now that $E_1\cdot B_2>0$. Then after three
blow-downs the vertical curve $G$ becomes a $0$-curve, so either (a)
$E_2$ intersects $G$ or (b) $E_1$ intersects a further $(-1)$-curve
in the chain (after it has been partially blown down). If $E_1$
intersects $B_2$ and $E_2$ intersects $G$ then none of the $E_i$
intersect the chain, and since the chain is nonempty, this provides a
contradiction.

Therefore $E_1$ should intersect the long chain, and it should
intersect it in the last curve to be blown down from there. Suppose
that $E_1\cdot C_i=1$. Then $E_1$ cannot intersect $D$, since
otherwise after blowing down $E_1$, then sequentially blowing down the
images of $B_2$ and $B_1$, $C_i '$ (the image of $C_i$) will intersect
the image of $D$ at least three times (counting with
multiplicity). When (the image of) $C_i '$ is eventually blown down,
the image of $D$ will gain a singularity which is not permitted for a
cubic in $\cpk$. This shows that $E_2$ has to intersect the chain
(and start the sequence of blow-downs) and it also has to intersect
$D$ to get a singularity on it.  Furthermore, we also know that $E_2$
must be disjoint from $G$.  The argument of Proposition~\ref{p:c6}
shows that $E_2$ must intersect the long chain at its farther end and
also $D$.  As usual, the framings are dictated by the fact that all
curves in the complement of the $(+1)$-curve must be blown down,
leading to the second set of values of $c$ and $c_i$. By determining
the dual graphs, the proof is complete.
\end{proof}

\smallskip
\noindent {\bf {The family $\frc _2$:}} The generic case in this
family is shown by Figure~\ref{f:c2}(a).
\begin{figure}[ht]
\begin{center}
\includegraphics[width=10cm]{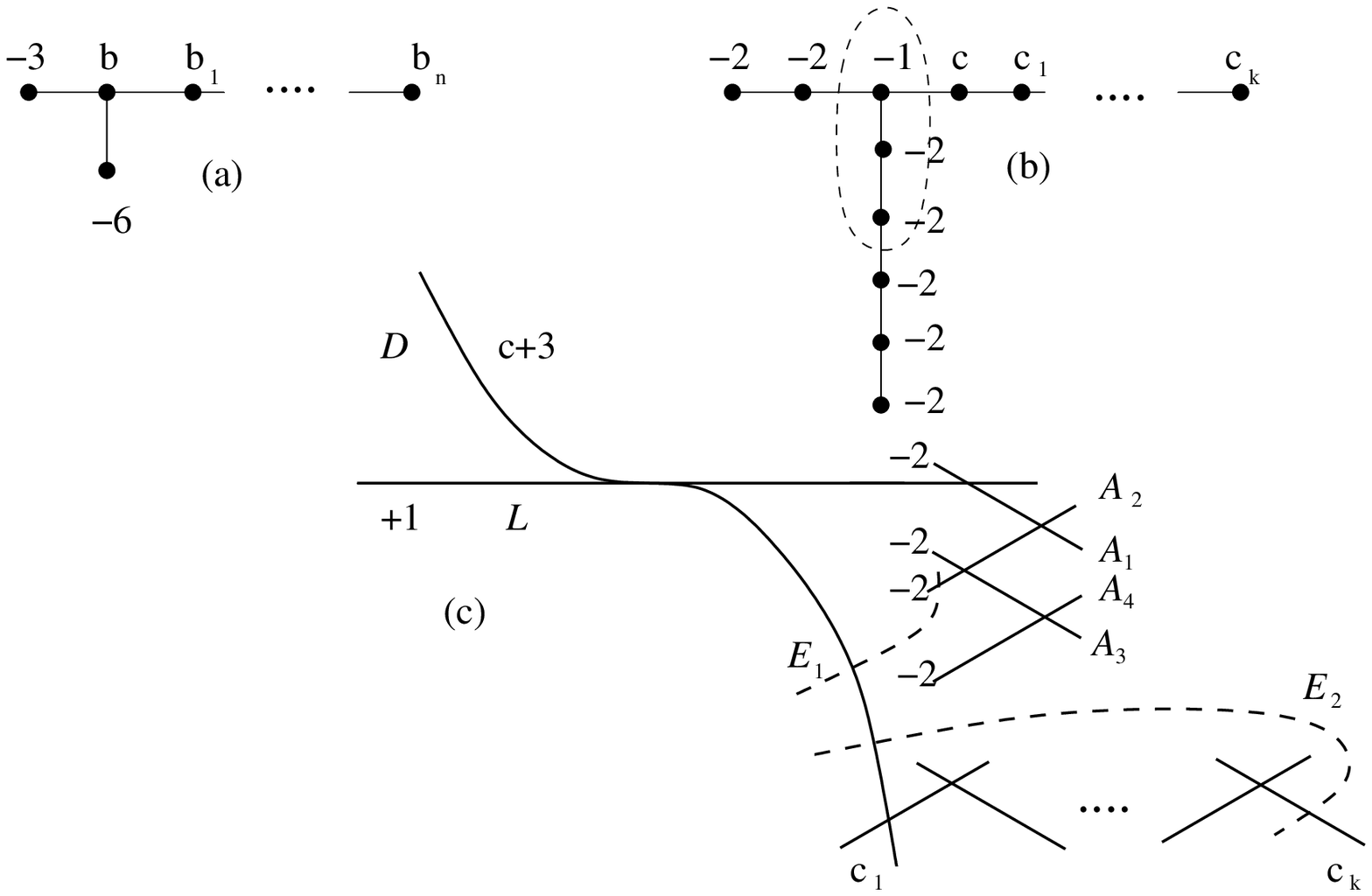}
\end{center}
\caption{The generic graph, its dual, and the configuration of curves after
3 blow-downs in the family $\frc _2$. The curves $E_1, E_2$ are shown only
for the first possibility given by Proposition~\ref{p:c2pr}.}
\label{f:c2}
\end{figure}
The usual simple calculation shows that by assuming the existence of a
$\bfq$HD filling for $(Y_{\G}, \xi _{\G} )$ we have to locate two
$(-1)$-curves in the diagram, which we will denote by $E_1$ and $E_2$.
Since the curves $A_2,A_3$ and $A_4$ must be blown down at some point
in the blow-down procedure, one of the $(-1)$-curves (say $E_1$)
should intersect $A_2\cup A_3\cup A_4$.
\begin{prop}\label{p:c2pr}
In the situation under examination, the existence of a $\bfq$HD
filling implies that $E_2$ intersects $D$ and $C_k$, while $E_1$
either intersects $A_2$ and $D$ or $A_4$ and $C_1$ or $A_4$ and $C_2$.
The framings in the three cases are given by $c=-k-2$ and $c_1=\ldots
c_k=-2$, or $c=-k+3$, $c_1=-5$, $c_3=-3$ and $c_2=c_4=\ldots =c_k=-2$,
or $c=-k+2$, $c_2=-6$ and $c_1=c_3=\ldots c_k=-2$.  In particular, the
resolution graph is as one of the graphs given by
Figure~\ref{f:qhd3}(i) in the first case, by
Figure~\ref{f:qhd3}(g) (with $p=0, r=2, q=k-4$) in the second,
and by Figure~\ref{f:qhd3}(e) ($p=3, q=k-4$) in the third.
\end{prop}
\begin{proof}
Notice first that $E_1$ cannot intersect $A_3$ (otherwise we will have
a self-intersection 0 curve in the complement of $L$, contradicting
Lemma~\ref{lem:nonexnonneg}); hence we have two cases to examine.

{\bf Case I:} Suppose that $E_1\cdot A_2>0$. In this case, after four
blow-downs, the self-intersection of $A_1$ becomes 1, which cannot go
any higher, since in $\cpk$ the curve $A_1$ will become a
line. Therefore $E_1$ must be disjoint from the chain and $E_2$ must
be disjoint from all the $A_i$'s. In order for the image of $A_1$ to
intersect $D$ three times, $E_1$ must intersect $D$. Since $E_2$ is
disjoint from all the $A_i$'s, and it starts the blow-down of the
chain, and is responsible for the singularity on $D$, the usual
argument presented in the proof of Proposition~\ref{p:c6} locates
it. In conclusion, the only possibility for the framings is the one
given by the statement.

{\bf Case II:} Suppose now that $E_1$ intersects $A_4$.  After blowing
down $E_1$, and then sequentially blowing down the images of $A_4,A_3$
and $A_2$, the self-intersection of $A_1$ will increase to $-1$. In
order to increase it to $1$ we have a number of possibilities.

(i) $E_1\cdot C_i=0$ for all $i$, i.e., $E_1$ is disjoint from the
chain.  In this case $E_2$ must intersect $A_1$ and also the last
curve we blow down in the chain.  Since then there is no further curve
starting the blow-down of the chain, this can happen only if the chain
has a single element. If $E_2$ is disjoint from $D$, then after all
blow-downs have been carried out $D$ remains smooth, which is a
contradiction.  Therefore $E_2$ must intersect $D$.  Blowing down
$E_2$ and then the elements in the chain we get that the image of
$A_1$ passes through $D$ three times. Therefore $E_1$ must be disjoint
from $D$. Computing the self-intersections, however, we see that the
curve with framing $c$ (giving rise to $D$, which will become of
self-intersection 9) must have self-intersection $c=1$ in the dual
graph, which is a contradiction.

(ii) Assume now that  $E_1$ intersects the chain in the curve we will
blow down last.  This implies that $E_2$ should intersect $A_1$, but
since the blow-down of $E_1$ (together with the last curve in the chain)
increases the self-intersection of $A_1$ by two, $E_2$ must be
disjoint from the chain.  Therefore, once again, the chain must be of
length one. Performing the blow-downs, we conclude that $D$ remains
smooth and the images of $D$ and $A_1$ will intersect each other only
twice, hence this case does not occur.

(iii) Finally, it can happen that $E_1$ intersects the chain in the
penultimate curve to get blown down.  Then $E_2$ should be disjoint
from the $A_i$'s, and since the singularity on $D$ cannot be caused by
blowing down $E_1$, we need that $E_2$ intersects $D$. The usual
argument given in the proof of Proposition~\ref{p:c6} shows the
position of $E_2$, leading to two configurations, depending on whether
the last curve to be blown down is next to $D$ or is one off. The
resulting framings in these two cases are then the ones given by the
proposition.
\end{proof}

\subsection{Graphs in $\fra$}
For three-legged graphs in $\fra$ there is no need for further
subdivisions since the legs in this case are symmetric.  As usual, the
generic member of the family is shown by Figure~\ref{f:a}(a).  The
usual simple count shows that if we assume the existence of a $\bfq$HD
filling, then we have to find two $(-1)$-curves $E_1, E_2$ in
Figure~\ref{f:a}(a). The curve $A$ is of self-intersection $(-2)$,
and will become a line in $\cpk$, hence must be hit by one of the
$(-1)$-curves, say by $E_1$.
\begin{figure}[ht]
\begin{center}
\includegraphics[width=10cm]{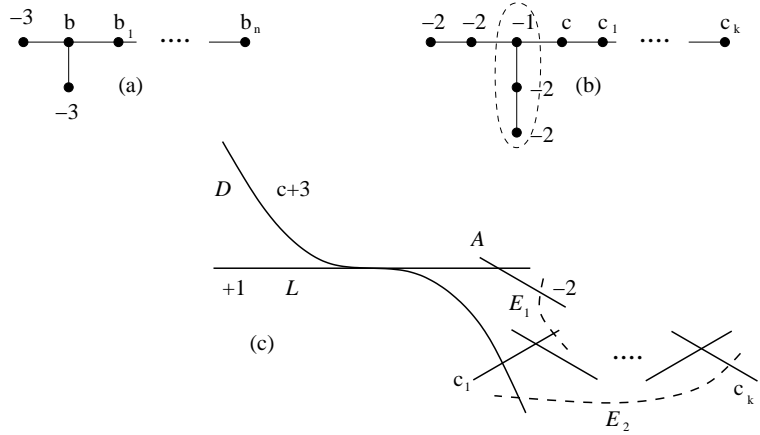}
\end{center}
\caption{The generic graph, its dual, and the configuration of curves after 3
  blow-downs in the family $\fra$. Another possibility for $E_1$ in (c)
  allowed by Proposition~\ref{p:casea} is the curve which intersects $A$ and
  $C_2$ (instead of $C_1$). As usual, we do not depict this second
  possibility.}
\label{f:a}
\end{figure}
\begin{prop}\label{p:casea}
In this case, the curve $E_2$ intersects $D$ and $C_k$, while $E_1$
intersects either $A$ and $C_1$ or $A$ and $C_2$. The corresponding
framings in both cases are $c=-k+2$, $c_2=-3$ and $c_1=c_3=\ldots
c_k=-2$.  In particular, the resolution graph is of the form of
Figure~\ref{f:qhd3}(e) with $p=0$.
\end{prop}
\begin{proof}
We are assuming that $E_1$ intersects $A$. If $E_2$ also intersects
$A$, then only one of them (say $E_2$) can intersect the long chain,
and only in the last curve to be blown down, so we cannot start the
blow-down process on the chain unless it is of length one. We show
that this case never occurs. In fact, to create the singularity on
$D$, the $(-1)$-curve $E_2$ must intersect it, and so by blowing down
$E_2$ and the unique element in the chain, we get that the resulting
$A$ and $D$ will intersect each other three times, hence $E_1$ must be
disjoint from $D$. The self-intersection of the resulting singular
cubic (which must be equal to 9) is $c+8$, implying that $c=1$, which
contradicts the fact that it should be negative.  Therefore $E_2$
cannot intersect $A$, and so it must intersect the long chain, and to
create the singular point on $D$ it must also intersect that
curve. The usual argument already discussed in Proposition~\ref{p:c6}
shows that $E_2$ can intersect the chain only in $C_k$. In order to
raise the self-intersection of $A$ from $-2$ to $1$ we need that
$E_1$ intersect the chain in the penultimate curve to be blown down.
Since after the blow-downs the image of $A$ will pass through the
singular point of $D$, $E_1$ must be disjoint from $D$.  The two very
similar possibilities for the $(-1)$-curves (differing only in the
position of the $E_1$-curve) result the same set of framings, hence
the same set of resolution graphs.
\end{proof}

\subsection{Graphs in $\frb$}
\label{ss:famb}
Similarly to the case of $\frc$, the study of three--legged graphs in
the family $\frb$ falls into two subcases, of $\frb _4$ and $\frb _2$,
depending on the choice of the first blow-up.  The family $\frb _2^3$
defined by (h) of Figure~\ref{f:qhd3}, for example, is a subfamily of
$\frb _2$.

\smallskip
\noindent {\bf{The family $\frb _4$:}} The generic member of this family
(together with the dual graph and the configuration of curves after
three blow-downs) is shown in Figure~\ref{f:b4}.
\begin{figure}[ht]
\begin{center}
\includegraphics[width=10cm]{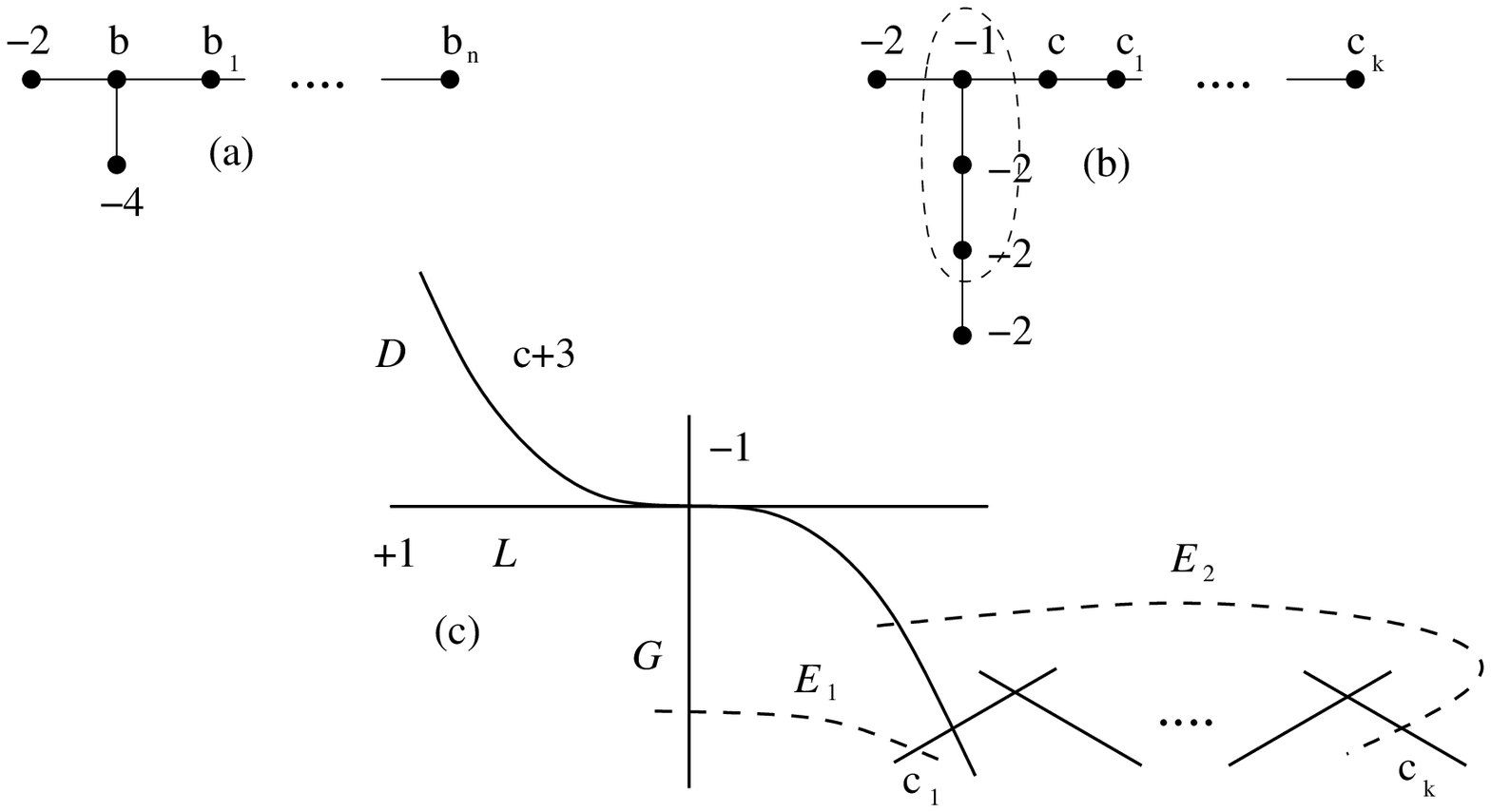}
\end{center}
\caption{The generic graph, its dual, and the configuration of curves after
3 blow-downs in the family $\frb _4$.}
\label{f:b4}
\end{figure}
The usual count of curves shows that we need to locate two
$(-1)$-curves, denoted by $E_1$ and $E_2$. It is clear that one of
them, say $E_1$, must intersect $G$ in order to increase its
self-intersection to $1$.
\begin{prop}\label{p:b4pr}
Under the above hypotheses, the existence of a $\bfq$HD filling
implies that $E_2$ intersects $D$ and $C_k$, while $E_1$ intersects
$G$ and $C_1$. The corresponding framings are $c=-k+2$, $c_1=-3$ and
$c_2=\ldots c_k=-2$. In particular, the resolution graph is of the
form given by Figure~\ref{f:qhd3}(d) with $r=0$.
\end{prop}
\begin{proof}
If $E_2$ also intersects $G$ then both $E_1$ and $E_2$ must be
disjoint from the chain, hence it cannot be blown down. Therefore we
can assume that $E_2$ is disjoint from $G$, and therefore $E_1$ must
intersect the chain in the last curve to be blown down. The curve
$E_1$ must be disjoint from $D$, since if $E_1$ intersects $D$ then
after two blow-downs the curves resulting from $G$ and $D$ will
intersect at least four times, giving a contradiction. Therefore $E_1$
must be disjoint from $D$, hence $E_2$ intersects the configuration of
curves as is found in the proof of Proposition~\ref{p:c6}.  The only
possibility for the framings is the one given by the proposition.
\end{proof}

\smallskip
\noindent {\bf{The family $\frb _2$:}}
The graphs (with their duals, and the curve configuration we get by
the three blow-downs) are shown in Figure~\ref{f:b2}. The usual curve
count shows that for identifying a $\bfq$HD filling we must find three
$(-1)$-curves $E_1, E_2, E_3$ in the diagram. Suppose that $E_1$
intersects $G$.
\begin{figure}[ht]
\begin{center}
\includegraphics[width=10cm]{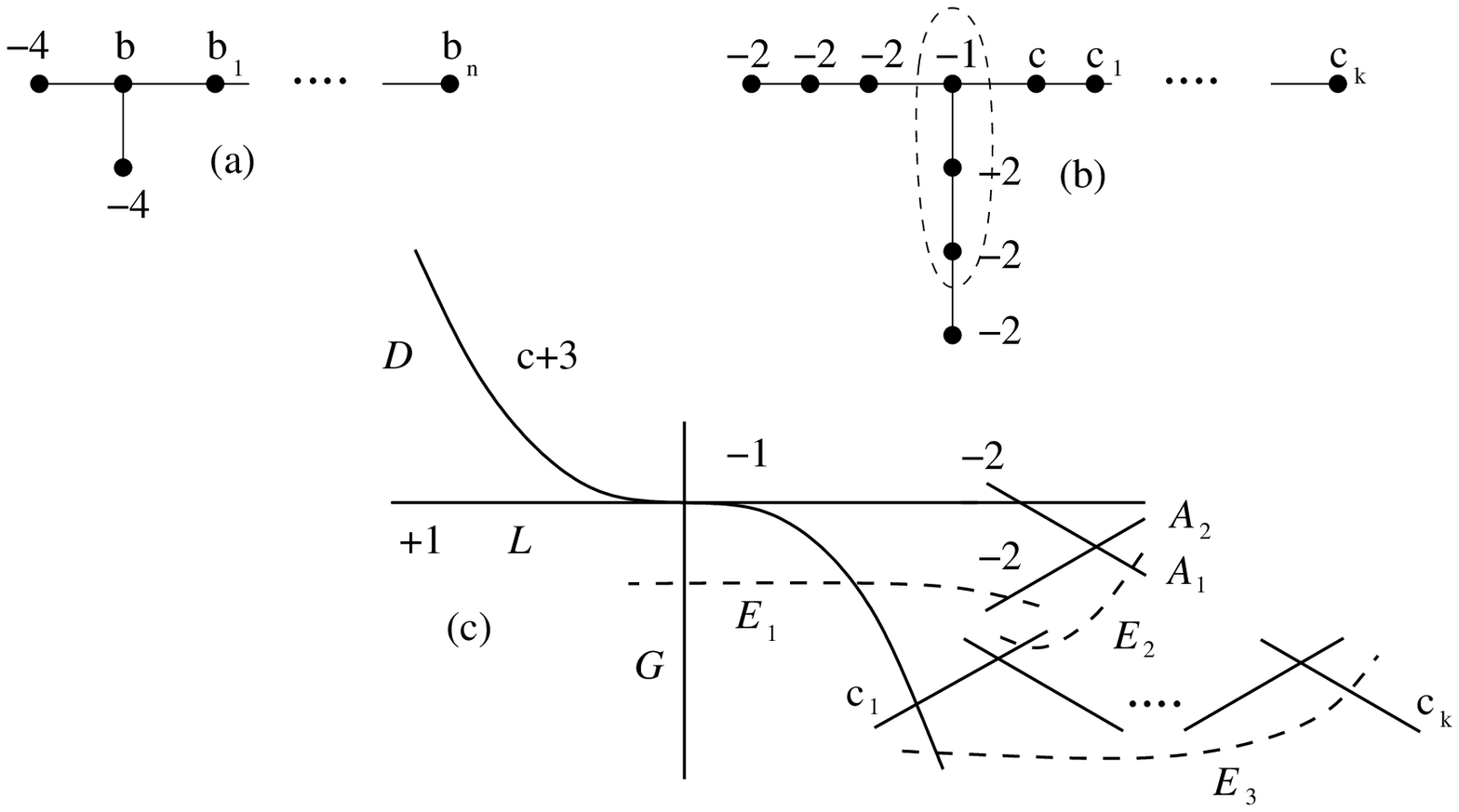}
\end{center}
\caption{The generic graph, its dual, and the configuration of curves
  after 3 blow-downs in the family $\frb _2$.  The curves $E_1, E_2$
  of (c) correspond to the first possibility listed by
  Proposition~\ref{p:caseb2}.}
\label{f:b2}
\end{figure}
\begin{prop}\label{p:caseb2}
Under the circumstance described above, from the existence of a
$\bfq$HD filling it follows that either
\begin{itemize}
\item the curve $E_3$ intersects $D$ and
$C_k$, $E_2$ intersects $C_1$ and $A_1$ and $E_1$ intersects
$G$, $D$ and $A_2$ and therefore the framings  satisfy
$c=-k$, $c_1=-3$ and $c_2=\ldots = c_k=-2$, or
\item $E_3$ intersects $D$ and $C_k$, $E_2$ intersects $A_2$ and $C_2$,
and $E_1$ intersects $G$ and $C_1$ and therefore the framings are given
as $c=-k+2$, $c_1=-3$, $c_2=-4$, $c_3=\ldots =c_k=-2$, or
\item $E_3$ intersects $D$ and $C_k$, $E_2$ intersects $A_2$ and $C_1$,
and $E_1$ intersects $G$ and $C_2$ and the framings are
$c=-k+2$, $c_1=-3$, $c_2=-4$, $c_3=\ldots =c_k=-2$.
\end{itemize}
In particular, the resolution graph is of the form of
Figure~\ref{f:qhd3}(h) in the first case and of
Figure~\ref{f:qhd3}(g) (with $p=1, r=0, q=k-4$) in the second and
third cases.
\end{prop}
\begin{proof}
Since $G$ has self-intersection $-1$ and it intersects the curve $L$
once, its self-intersection must increase to $1$, hence either another
$(-1)$-curve, say $E_2$, intersects $G$ or $E_1$ intersects either
$A_2$ or the chain. Note that $E_1$ cannot intersect both $A_2$ and
the chain, since if $E_1\cdot C_i=1$ and $E_2\cdot A=1$, then after
$E_1$ and the image of $A_2$ are blown down the image of $C_i$ will
become tangent to the image of $G$. When the image of $C_i$ is
eventually blown down, the image of $G$ will gain a singularity, which
is impossible for a line in $\cpk$.

\noindent {\bf Case I:} $E_2\cdot G=1$.  In this case both $E_1$ and
$E_2$ must be disjoint from $A_2$ and the chain, hence $E_3$
intersects both $A_2$ and the chain.  Also, since $G$ and $A_1$ will
intersect after the blowing down process has been carried out, $E_1$
or $E_2$ (say $E_1$) must intersect $A_1$. After blowing down the
$E_i$'s and the image of $A_2$, the self-intersection of $A_1$ will
already be zero, hence $E_3$ can only intersect the chain in the last
curve to get blown down, which is possible only if the chain is of
length one.  If $E_3$ is disjoint from $D$ then (in order for $A_1$ to
intersect $D$ three times) $E_1$ must intersect $D$ twice, and hence
(in order to avoid $G\cdot D>3$) the curve $E_2$ must be disjoint from
$D$. Now we can easily see that the self-intersection of $D$ increases
to $c+8$ after all the blow-downs have been performed, and since it
should be equal to 9, we deduce that $c=1$, contradicting the fact
that $c$ is negative. If $E_3$ intersects $D$ then after blowing
down $E_3$ and then sequentially blowing down the images of $A_2$ and
the unique element in the chain we get a singularity on $D$ of
multiplicity $3$, a contradiction. This shows that Case I, in fact,
cannot occur.

\noindent {\bf Case II:}
$E_1\cdot A_2=1$. Then both $E_2$ and $E_3$ must be disjoint from $G$, and one of them
(say $E_2$) intersects $A_1$. To increase the self-intersection of
$A_1$, the curve $E_2$ should intersect the chain in the last curve to
be blown down.  Since the image of $G$ will intersect $D$, we see that
$E_1\cdot D=1$.  This implies that after blowing down $E_1$ and $A_2$,
the curve $A_1$ will intersect $D$ once, therefore $E_2$ cannot
intersect $D$ (since it would add three to $A_1 \cdot D$). Now the
usual argument from the proof of Proposition~\ref{p:c6} shows that
$E_3$ starts the blow-down of the chain, and it also intersects $D$
in one point, leading to the first case of the proposition.

\noindent {\bf Case III:} $E_1\cdot C_i=1$.  Recall that by the
previous argument we can assume that $E_2\cdot G= E_3\cdot G=0$. If
$E_2$ and $E_3$ are both disjoint from the chain, then the chain must
have length one. But then, if $E_1\cdot D=0$, then, after completing
the blowing down process, the intersection number of the images of $G$
and $D$ will be less than $3$ and if $E_1\cdot D=1$, then, after
completing the blowing down process, the intersection number of the
images of $G$ and $D$ will be greater than $3$, both contradicting the
fact that the intersection number of a line and a cubic in $\cpk$ is
equal to three. So we may assume that $E_3$ intersects the chain, say
$E_3\cdot C_l=1$, and, by the preceding argument, that $E_1\cdot
D=0$. If $E_3\cdot D=0$, again we find that, after the blowing down
process has been carried out, the intersection number of the images of
$G$ and $D$ will be $2$, a contradiction. So we must have $E_3\cdot
D=1$. Now observe that we must have $E_3\cdot A_2=0$. Indeed, if
$E_3\cdot C_j=1$ and $E_3\cdot A_2=1$, then after $E_3$ and the image
of $A_2$ are blown down, the image $C_j '$ of $C_j$ will be tangent to
the image $D$. It is now easy to see that after the blowing down
process is complete the image of $D$ will have more than one singular
point or a singularity of multiplicity greater than $2$, both of which
are impossible for a cubic in $\cpk$. Since $A_2$ must be hit by a
$(-1)$-curve, we deduce that $E_2\cdot A_2=1$. We now check that $E_1$
and $E_3$ are disjoint from $A_1$. If $E_1\cdot A_1=1$, then after
blowing down $E_1$ the images of $G$ and $A_1$ will intersect in a
point and the image of $C_i$ will pass through that point. When the
image of $C_i$ is eventually blown down, the intersection number of
the images of $G$ and $A_2$ will be $2$, which is impossible for a
pair of lines in $\cpk$. If $E_3\cdot A_1=1$, then the chain must have
length one (to prevent the intersection number of the images of $A_1$
and $D$ going above $3$). Usual simple calculation shows that $c$ must
be $1$ contradicting $c<0$. We have thus checked that $E_1$ and $E_3$
are disjoint from $A_1$. It follows that, in order for the
self-intersection number of the image of $A_1$ to increase to $1$, we
must have that $E_2$ intersects the string in the penultimate curve of
the chain to get blown down. Suppose that $E_2\cdot C_j=1$. Now if
$l<k$, then it is easy to see that we must have $k=2,\ l=1$ and
$j=2$. But then, after completing the blowing down process, the
intersection number of the images of $A_1$ and $D$ will be $2$, a
contradiction. Thus we must have $l=k$. It follows that we must have
$j=1$ or $2$. If $j=1$, then we must have $i=2$, and if $j=2$, then we
must have $i=1$. The blowing down process now fixes $c,
c_1,\ldots,c_k$, which depends only on $k$ and is independent of $j$,
giving $c=-k+2, c_1=-3, c_2=-4$ and $c_3=\ldots =c_k=-2$. The two
possible configurations of the curves $E_1, E_2, E_3$ (providing the
same dual graphs) are the ones given by the proposition.
\end{proof}

\begin{proof}[Proof of Theorem~\ref{t:main}]
Consider a small Seifert singularity $S_{\G}$. Since a smoothing of
$S_{\G}$ provides a weak symplectic filling of the Milnor fillable
contact structure $(Y_{\G} , \xi _{\G })$ of the link, the implication
$(1)\Rightarrow (2)$ follows. The implication $(2)\Rightarrow (3)$ is
a direct consequence of the combination of Propositions~\ref{p:c6},
\ref{p:c3pr}, \ref{p:c2pr}, \ref{p:casea}, \ref{p:b4pr} and
\ref{p:caseb2}, together with Theorem~\ref{t:ssw}.

In order to verify the implication $(3)\Rightarrow (1)$, we need to
produce $\bfq $HD smoothings for singularities with resolution graphs
in ${\mathcal {QHD}}^3$. This result follows from
\cite[Example~8.4]{SSW} for the graphs of Figure~\ref{f:qhd3}(a), (b)
and (c), and from \cite[Example~8.3]{SSW} for (d), (e), (f) and
(g). For singularities with resolution graphs given by
Figure~\ref{f:qhd3}(h), (i) and (j) we give an argument resting on the
theorem of Pinkham \cite{Pink} as formulated in
\cite[Theorem~8.1]{SSW}. In order to apply this result for a
singularity $S_{\G}$, we need to find an embedding of rational curves
in a rational surface $R$ intersecting each other according to the
dual graph $\G '$ with the property that rk$H_2(R; \bfz )=\vert \G
'\vert$.
\begin{figure}[ht]
\begin{center}
\includegraphics[width=8cm]{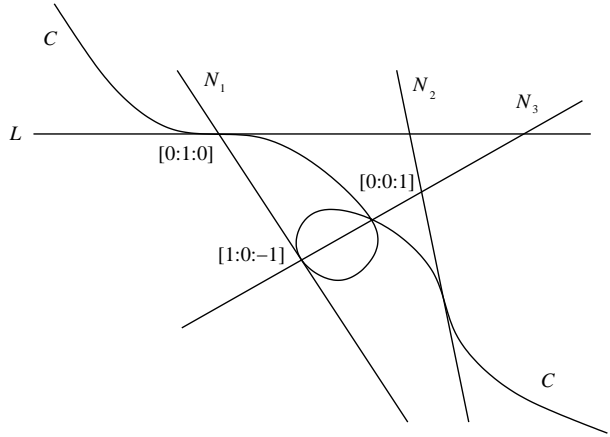}
\end{center}
\caption{The curves used in the constructions of the embeddings.}
\label{f:curves}
\end{figure}

To this end, let us consider the singular cubic $C$ given by equation
$f(x,y,z)=y^2z-x^3-x^2z$ in $\cpk$ and the lines $L, N_1, N_2$ and
$N_3$ given by the equations $\{ z=0\}$, $\{ x+z=0\}$, $\{
y-(x+\frac{8}{9}z)\sqrt{3}i=0\}$, and $\{ y=0 \}$, respectively,
cf. Figure~\ref{f:curves}.  (The line $L$ is tangent to $C$ at one of
its inflection point $[0:1:0]$; $N_1$ is tangent to $C$ at $[-1:0:1]$
and further intersects it in $[0:1:0]$; $N_2$ is tangent to $C$ in
another inflection point $[-\frac{4}{3}:-i\frac{4}{3\sqrt{3}}:1]$.)

By sequentially blowing down the $(-1)$--curves of Figure~\ref{f:c6}(c)
(starting with the dashed one), we are led to a configuration of
curves involving a singular cubic and a tangent at one of its
inflection points. Since $C$ and $L$ provide such a configuration, the
reverse of the above blow-down procedure gives an embedding of the
configuration of Figure~\ref{f:c6}(c), and therefore of (b) into some
blow-up of $\cpk$. A simple count of the applied blow-ups shows that
this embedding is exactly of the type needed to apply Pinkham's
result, hence this argument shows that graphs of
Figure~\ref{f:qhd3}(f) correspond to singularities with $\bfq$HD
smoothings. The same line of argument, with various starting
configurations, then shows that all the remaining graphs of
Figure~\ref{f:qhd3} correspond to singularities with $\bfq$HD
smoothings: a suitable starting configuration for the graphs of
Figure~\ref{f:qhd3}(h) is the configuration given by the curves $L, C,
N_1$ and $ N_3$ of Figure~\ref{f:curves}, for (i) $L,C, N_2$ and for
(j) $L, C, N_1$ will be a convenient choice. With this last step, the
proof of Theorem~\ref{t:main} is now complete.
\end{proof}

\section{Spherical Seifert singularities}
\label{s:fourleg}
Next we turn to the examination of generic spherical Seifert
singularities.  Since a star-shaped graph in $\fra \cup \frb \cup \frc$
can have at most four legs, it follows from Theorem~\ref{t:ssw} that
if a spherical Seifert singularity admits a $\bfq$HD smoothing (or the
Milnor fillable contact structure on its link admits a $\bfq$HD
filling) then the valency of the central vertex is at most four. The
three-legged graphs were analyzed in the previous section, so now we
will focus on the case of four-legged graphs. Once again, it follows
from Theorem~\ref{t:ssw} that we only need to consider graphs in $\fra
\cup \frb \cup \frc$.

\subsection{The family $\frc$}
We start by considering the four-legged graphs in the family $\frc$.
\begin{figure}[ht]
\begin{center}
\includegraphics[width=10cm]{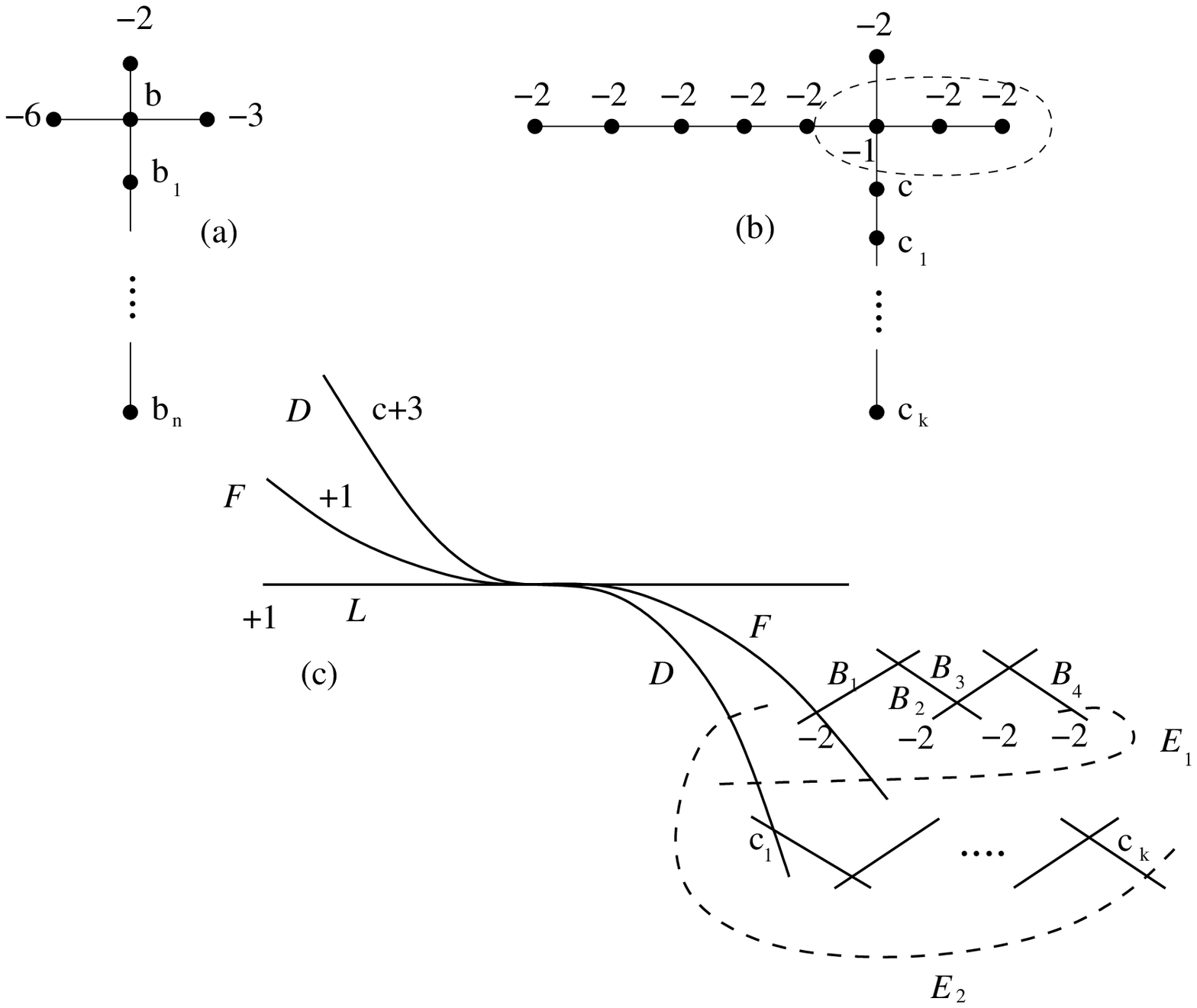}
\end{center}
\caption{The four-legged graphs in $\frc$.}
\label{f:c4leg}
\end{figure}
The generic four-legged member $\G$ of $\frc$ is given in
Figure~\ref{f:c4leg}(a), with the dual graph given by
Figure~\ref{f:c4leg}(b). After three blow-downs we obtain the
configuration $K$ depicted in Figure~\ref{f:c4leg}(c). As before, the
horizontal $(+1)$-curve will be denoted $L$ and the two curves which
are triply tangent to $L$ will be denoted $F$ and $D$, with $F$ being
the curve with square $+1$. The chain of $(-2)$-curves connected to
the curve $F$ will be denoted $B_1,\ldots,B_4$, with $B_1$
intersecting $F$ and the chain of curves intersecting $D$ will be
denoted $C_1,\ldots,C_k$, with $C_1$ intersecting $D$.  By
symplectically gluing $K$ to a $\bfq$HD filling $X$ we get a closed
symplectic 4-manifold $Z$, and the usual elementary homological
computation shows that (since $X$ is a $\bfq$HD) there must be
precisely two $(-1)$-curves, say $E_1$ and $E_2$, in the complement of
$L$ that are not contained in the strings $B_1,\ldots,B_4$ and
$C_1,\ldots,C_k$. Since the string $B_1,\ldots,B_4$ must be
transformed into a configuration which can be sequentially blown down
after blowing down $E_1$ and $E_2$, it follows that at least one of
these $(-1)$-curves must intersect $B_1\cup\cdots\cup B_4$. Assume,
without loss of generality, that $E_1$ intersects $B_1\cup\cdots\cup
B_4$.

\begin{prop}\label{p:c4leg}
By assuming the existence of the $\bfq$HD filling $X$ we get that
$E_1$ intersects $D$, $F$ and $B_4$, while $E_2$ intersects $D$ and
$C_k$. The framings then are given by $c=-k-3$ and $c_1=\ldots
=c_k=-2$ (with $k\geq 0$). In particular, the graph of
Figure~\ref{f:c4leg}(a) should be of the form
Figure~\ref{f:qhd4uj}(c).
\end{prop}
\begin{proof}
If $E_1\cdot B_2=1$ or $E_1\cdot B_3=1$, then blowing down $E_1$ and
then sequentially blowing down the images of $B_2$ and $B_3$ leads to
a $(+1)$-curve (the image of $B_1$ or $B_4$) in the complement of $L$,
contradicting Lemma~\ref{lem:nonexnonneg}. Hence we can assume that
either $E_1\cdot B_1=1$ or $E_1\cdot B_4=1$.  First we argue that
$k>0$ can be assumed in Figure~\ref{f:c4leg}(b).  Indeed, $k=0$
implies that in Figure~\ref{f:c4leg}(a) we have $b=-3$, $b_1=\ldots =
b_n=-2$. Among these possibilities only the one with $n=2$ is in
$\frc$, and that particular graph appears among the ones of
Figure~\ref{f:qhd4uj}(c). For this reason, from now on we will assume
that $k>0$.

\medskip
{\bf Case I:} Suppose that $E_1\cdot B_1=1$.  Note first that
$E_1\cdot F=0$. Indeed, suppose that $E_1\cdot F\geq 1$. If $E_1\cdot
F>1$, then blowing down $E_1$ would lead to a point on the image
$F^\prime$ of $F$ under the blowing down map through which at least
two branches of $F^\prime$ pass. Also the intersection number of the
image $B_1^\prime$ of $B_1$ and $F^\prime $ will be at least three. By
perturbing the almost complex structure slightly, we can assume that
$B_1^\prime$ and $F^\prime$ intersect transversely. Then blowing down
$B_1^\prime$ we see that the image $F^{\prime\prime}$ of $F^\prime$
will have two singularities, which by Lemma~\ref{l:sing} contradicts
the fact that $F^{\prime\prime}$ will eventually blow down to a cubic
in $\cpk$. A similar contradiction arises if $E_1\cdot F = 1$, after
blowing down both $E_1$ and $B_1^\prime$. There are now two
possibilities: $E_1\cdot (C_1\cup\cdots\cup C_k)=1$ or $E_1\cdot
(C_1\cup\cdots\cup C_k)=0$. Note that $E_1\cdot (C_1\cup\cdots\cup
C_k)>1$ is impossible by Corollary~\ref{c:nocycle}.

\medskip
{\bf {IA.}} $E_1\cdot (C_1\cup\cdots\cup C_k)=1$.
Suppose that $E_1\cdot C_i=1$. After blowing down $E_1$ and then
sequentially blowing down the images of $B_1,\ldots,B_4$, observe that
the image $C_i^\prime$ of $C_i$ will be $4$-fold tangent to the image
$F^\prime$ of $F$. Perturbing the almost complex structure, we may
assume that $C_i^\prime$ intersects $F^\prime$
transversely. Eventually $C_i^\prime$ will get blown down and this
will create a singularity on the image of $F$ that is not allowed for
a cubic in $\cpk$, since the link of its singularity has four
components, providing the desired contradiction.

\medskip
{\bf {IB.}} $E_1\cdot (C_1\cup\cdots\cup C_k)=0$.  We have $E_1\cdot
D=0$ or $E_1\cdot D=1$. ($E_1\cdot D>1$ is not allowed as blowing down
$E_1$, then perturbing the almost complex structure so that
$B_1^\prime$, the image of $B_1$, and $D^\prime$, the image of $D$,
intersect transversely and then blowing down $B_1^\prime$ would create
two nodes on the image of $D^\prime$, contradicting
Lemma~\ref{l:sing}.) After blowing down $E_1$ and then sequentially
blowing down the images of $B_1,\ldots,B_4$, the intersection number
of the images $F^\prime$ and $D^\prime$ of $F$ and $D$, respectively,
will be either $3$ or $7$. Now, by arguing as in the proof of
Proposition~\ref{p:c6}, we can show that $E_2$ must intersect the last
curve $C_k$ in the string $C_1,\ldots,C_k$ and the curve
$D^\prime$. $E_2$ must also intersect $F^\prime$, otherwise, after the
blowing down process has been carried out, the image of $F^\prime$
would be nonsingular and rational, which is impossible for a cubic in
$\cpk$. In fact, it is necessary that $E_2\cdot F^\prime=2$, otherwise
the image of $F^\prime$ will either be smooth or have the wrong type
of singularity. Also it is necessary that the string $C_1,\ldots,C_k$
be empty, otherwise, after blowing down $E_2$, when the image of $C_k$
is collapsed a further singularity will be introduced in the image of
$F^\prime$. Now the condition that $D^\prime$ gets blown down to a
rational cubic in $\cpk$ forces us to have $E_2\cdot
D^\prime=2$. Blowing down $E_2$, we see now that the intersection
number of the images of $D^\prime$ and $F^\prime$ will be either $7$
or $11$ (depending on $E_1\cdot D =0 $ or $1$), which is impossible
for a pair of irreducible cubic curves in $\cpk$. In conclusion, we
found that $E_1\cdot B_1=1$ leads to contradiction, hence we can
consider

\medskip
{\bf Case II:} $E_1\cdot B_4=1$.
As before, we distinguish two cases according to the intersection of
$E_1$ with the chain $C_1 \cup \ldots \cup C_k$.

\medskip
{\bf {IIA.}} $E_1\cdot (C_1\cup\cdots\cup C_k)=1$.
Suppose that $E_1\cdot C_i=1$. Note that $E_1\cdot F=0$, otherwise the
image of $F$ after completing the blowing down process would have more
than one singular points. For a similar reason, $E_1\cdot D$ must also
be $0$. We now divide $E_1\cdot C_i=1$ into three cases.

\medskip
{(i)} Suppose that $i=1$, i.e., $E_1$ intersects the chain in the
curve intersecting $D$. Blow down $E_1$, then sequentially blow down
the images of $B_4,\ldots,B_1$ and then the images of $C_1,\ldots,C_l$
until the resulting string $C_{l+1}^\prime,\ldots,C_k^\prime$ attached
to $D^\prime$, the image of $D$, is minimal, that is, contains no
$(-1)$-curves. Let $F^\prime$ denote the image of $F$. Then
$F^\prime\cdot D^\prime=l+2$, where $0\leq l\leq k$. First suppose that
$l<k$. Then, by arguing as in the proof of Proposition~\ref{p:c6}, one
can show that $E_2$ must intersect the last curve $C_k^\prime$ of the
string $C_{l+1}^\prime,\ldots,C_k^\prime$ and the curves $F^\prime$
and $D^\prime$, each once transversally. Now blow down $E_2$ and then
sequentially blow down the images of
$C_k^\prime,\ldots,C_{l+1}^\prime$. Then the images of $F^\prime$ and
$D^\prime$ will be nodal curves and for the intersection number of
them to be $9$ we require that $k=2$. However, to make the
self-intersection number of the image of $F^\prime$ equal $9$ we
require that $k=3$. This contradiction show that the case $l<k$ cannot
occur.  Now suppose that $l=k$. Then to introduce singularities of the
right type into the images of the curves $F^\prime$ and $D^\prime$ we
require that $E_2\cdot F^\prime=2$ and $E_2\cdot D^\prime=2$. A simple
check now shows that, as before, to make the intersection number of
the images of $F^\prime$ and $D^\prime$ $9$ we require $k=2$ and to
make the image of $D^\prime$ have self-intersection number $9$ we
require $k=3$, again a contradiction.

\medskip
{(ii)} Suppose next that $1<i<k\ (k\geq 3)$.  Blow down $E_1$, then
sequentially blow down the images of $B_4,\ldots,B_1$. Suppose first
that the image $C_i^\prime$ of $C_i$ under the blowing down map is not
a $(-1)$-curve.  Then, arguing as in the proof of
Proposition~\ref{p:c6}, one can show that $E_2$ must intersect the
last curve $C_k$ in the string attached to $D$ and it must necessarily
intersect $F^\prime$, the image of $F$. It follows that $i=1$,
otherwise, after blowing down $E_2$ and then sequentially blowing down
the images of $C_k,\ldots,C_1$, the image of $F^\prime$ would have
more than one singularity, contradicting Lemma~\ref{l:sing}. Since
$i>1$ is assumed, we reached a contradiction. Thus $C_i^\prime$ must
be a $(-1)$-curve. Now blow down $C_i^\prime$. Note that the images of
the curves $C_{i-1}$ and $C_{i+1}$ must be the last two curves (in
some blowing down process) of the string attached to $D$ to get blown
down, otherwise the image of $F^\prime$ after completing the blowing
down process will have more than one singular point, a
contradiction. Now there are two cases to consider: $E_2\cdot
F^\prime=0$ or $E_2\cdot F^\prime=1$.

Suppose that $E_2\cdot F^\prime=0$. Then it is easy to see that after
the blowing down process has been carried out, the image of $F^\prime$
will have self-intersection number $8$, which contradicts the fact
that $F$ should blow down to a cubic in $\cpk$.

Suppose that $E_2\cdot F^\prime=1$. Then $E_2$ must be disjoint from
the string attached to $D$. In order to make $D$ singular, $E_2\cdot
D$ must necessarily be $2$. It is now easy to check that, after
carrying out the blowing down process, the intersection number of the
images of the curves $F$ and $D$ will be less than $9$, which
contradicts the fact that they should blow down to a pair of cubics in
$\cpk$.

\medskip
{(iii)} Finally assume that $i=k\ (k\geq 2)$. Blow down $E_1$, then
sequentially blow down the images of $B_4,\ldots,B_1$ and then the
images of $C_k,\ldots,C_{l+1}$ until the resulting string
$C_1^\prime,\ldots,C_{l}^\prime$ attached to $D^\prime$ (the image of
$D$) is minimal. If a nonempty string remains, then, as before, $E_2$
must intersect the last curve $C_{l}^\prime$ in the string and the
curves $F^\prime$, the image of $F$, and $D^\prime$, each once
transversally. Then blowing down $E_2$ and then the image of
$C_l^\prime$, we find that $l$ must be $1$, otherwise the image of
$F^\prime$, after completing the blowing down process, would have more
than one singular point, contradicting Lemma~\ref{l:sing}. It follows
that the intersection number of the images of $F^\prime$ and
$D^\prime$, after completing the blowing down process, will be $8$,
contradicting the fact that they should blow down to a pair of cubics
in $\cpk$.

If $l=1$, that is the whole string attached to $D$ gets sequentially
blown down after blowing down $E_1$, then one can check that the
intersection numbers of $E_2$ and the images of $F^\prime$ and
$D^\prime$ must both be $2$. Again it follows that, after completing
the blowing down process, the intersection numbers of the images of
$F^\prime$ and $D^\prime$ will be $8$, a contradiction. This completes
{\bf {IIA}} and hence we conclude that

\medskip
{\bf {IIB.}} $E_1\cdot (C_1\cup\cdots\cup C_k)=0$.
We claim that $E_1\cdot F=1$. To see this, suppose, for a
contradiction, that $E_1\cdot F=0$. Then we have $E_1\cdot D=0$ or
$1$. Blow down $E_1$ and then sequentially blow down the images of the
curves $B_4,\ldots,B_1$. Then the image $F^\prime$ of $F$ will still
be smooth. It is thus necessary to have $E_2\cdot F^\prime=2$,
otherwise the image of $F$ will be smooth or have the wrong type of
singularity. But then the string $C_1,\cdots,C_k$ must be empty,
otherwise $E_2$ would have to intersect it and thus blowing down would
create additional singular points on the image of $F$, a
contradiction. It follows that, after completing the blowing down
process, the intersection number of the images of $F$ and $D$ will be
less than $9$, a contradiction. This verifies $E_1\cdot F=1$.

Now blowing down $E_1$ and then sequentially blowing down the $B_i$,
we find that the image of $F$ becomes a rational curve with a single
nodal point and having self-intersection number $9$. It follows that
$E_2$ cannot intersect $F$ and that $E_1$ must intersect $D$ once
transversally. Let $F^\prime$, $D^\prime$ denote the images of $F$ and
$D$, respectively, after blowing down $E_1$ and the $B_i$. It is then
easy to check that $F^\prime\cdot D^\prime=9$. Now the only
possibility for $E_2$, by the argument in the proof of
Proposition~\ref{p:c6}, is that $E_2\cdot C_k=1$ and $E_2\cdot
D=1$. For each value of $k$, the blowing down process now fixes $c$
and $c_1,\ldots,c_k$, providing the result.
\end{proof}

\subsection{The family $\frb$}
We next consider four-legged graphs in the family $\frb$:
\begin{figure}[ht]
\begin{center}
\includegraphics[width=10cm]{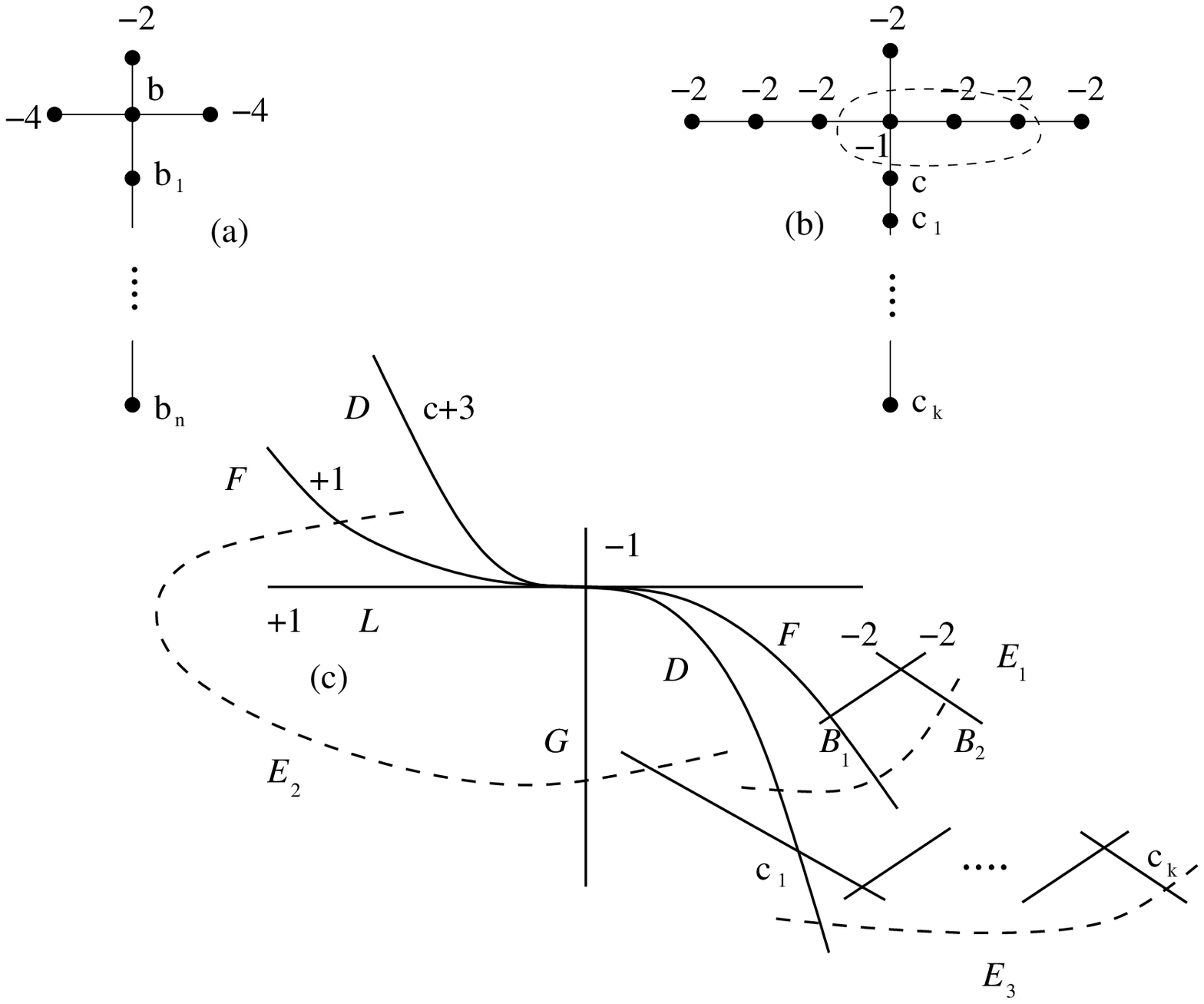}
\end{center}
\caption{The four-legged graphs in $\frb$.}
\label{f:b4leg}
\end{figure}
the generic four-legged member of this family is given by
Figure~\ref{f:b4leg}(a) with the dual graph given by
Figure~\ref{f:b4leg}(b).  After three blow-downs we obtain the
configuration $K$ depicted in Figure~\ref{f:b4leg}(c). As before, $Z$
is the closed symplectic $4$-manifold we get by gluing the
compactifying divisor $W_{\G '}$ (containing $K$) to a weak symplectic
$\bfq$HD filling of $(Y_\G , \xi _{\G})$. It is easy to check that
there must be three $(-1)$-curves, say $E_1,E_2,E_3$, not contained in
the strings $B_1,B_2$ and $C_1,\ldots,C_k$, such that, after blowing
down these three $(-1)$-curves, the images of the curves in the
strings $B_1,B_2$ and $C_1,\ldots,C_k$ can be sequentially blown down
and in the process $F$ and $D$ will be transformed to a pair of cubics
in $\cpk$ and the images of $G$ and $L$ will be lines. Since in the
blowing down process the string $B_1,B_2$ will eventually transform
into a string which can be sequentially blown down, one of the
$(-1)$-curves $E_1,E_2,E_3$, must intersect $B_1\cup B_2$; assume that
this curve is $E_1$.

\begin{prop}\label{p:b4leg}
Under the hypothesis of the existence of a $\bfq$HD filling, we get
that $E_1$ intersects $D$, $F$ and $B_2$, $E_2$ intersects $F$, $G$
and $C_1$, while $E_3$ intersects $D$ and $C_k$. The corresponding
framings are given as $c=-k-2$, $c_1=-3$ and $c_2=\ldots = c_k=-2$.
In particular, the resolution graph is of the form given by
Figure~\ref{f:qhd4uj}(b).
\end{prop}
\begin{proof}
Note that $E_1$ must be disjoint from $G$, otherwise blowing down
$E_1$ and then sequentially blowing down the images of $B_1$ and $B_2$
the image of $G$ would be either singular or would have
self-intersection number $2$, which contradicts the fact that $G$
should blow down to a line in $\cpk$. Since one of the $E_i$ must
necessarily intersect $G$ we may assume that $E_2\cdot G=1$.  We now
consider the two possibilities: $E_1\cdot B_i=1$ for $i=1,2$.

\medskip
{\bf Case I:} $E_1\cdot B_1=1$. The curve $E_1$ must necessarily be
disjoint from $F$, otherwise the image of $F$ after completing the
blowing down process would have more than one singular point which is
impossible for a cubic in $\cpk$.  We consider the two possibilities:
$E_1\cdot (C_1\cup\cdots\cup C_k)=1$ or $E_1\cdot (C_1\cup\cdots\cup
C_k)=0$.

\medskip
{\bf {IA.}} $E_1\cdot (C_1\cup\cdots\cup C_k)=1$.
Suppose that $E_1\cdot C_i=1$. Note that the image of $C_i$ must be
the last curve of the string attached to $D$ to get blown down, since
blowing down the the image of $C_i$ will make the image of $F$
singular so that if there are any remaining curves in the string then
these will create additional singularities on the image of $F$ when
they are blown down, a contradiction.

Suppose that $E_2\cdot (C_1\cup\cdots\cup C_k)=0$. Then the condition
that $G$ blows down to a $(+1)$-curve in $\cpk$, forces us to have
$E_3\cdot G=1$. But then necessarily $E_3\cdot (C_1\cup\cdots\cup
C_k)=0$. Thus the string $C_1,\ldots,C_k$ must have length $1$. Now
$E_2$ and $E_3$ must necessarily intersect $F$, each once
transversally, otherwise the intersection number of the images of $F$
and $G$ will not be $3$. It is also necessary that the intersection
number of one of $E_2$ or $E_3$ and $D$ be $2$ and the other be $0$ to
meet the requirements that the image of $D$ be singular and that the
images of $D$ and $G$ have intersection number $3$.  But then after
completing the blowing down process we will find that the images of
$D$ and $F$ have intersection number $7$, a contradiction.

Suppose that $E_2\cdot (C_1\cup\cdots\cup C_k)=1$. Note that $E_2$
must necessarily intersect $C_i$, the last curve in the string to get
blown down, otherwise the image of $G$ after repeatedly blowing down
will have self-intersection number greater than $1$, a
contradiction. Note also that $E_2$ must be disjoint from $F$,
otherwise blowing down the image of $C_i$ will lead to a triple point
on the image of $F$, a contradiction. Now consider the $(-1)$-curve
$E_3$. If $E_3$ intersects $C_1\cup\cdots\cup C_k$, then $E_3$ will be
disjoint from $F$. In such a case, after completing the blowing down
process, the image of $F$ will be a $7$-curve, a contradiction. If
$E_3$ is disjoint from $C_1\cup\cdots\cup C_k$, then $E_3\cdot F$ can
be $0$ or $1$. In either case, after completing the blowing down
process, the image of $F$ will have self-intersection number at most
$8$, again a contradiction. This argument concludes the analysis of the
case $E_1\cdot (C_1\cup\cdots\cup C_k)=1$.

\medskip
{\bf {IB.}} $E_1\cdot (C_1\cup\cdots\cup C_k)=0$.
Suppose that $E_2\cdot (C_1\cup\cdots\cup C_k)=0$ as well. As before, it
implies that $E_3\cdot G=1$. It follows that $E_1,E_2,E_3$ will be disjoint
from $C_1\cup\cdots\cup C_k$. But this means that the string must be empty,
which is never the case.

Suppose now that $E_2\cdot (C_1\cup\cdots\cup C_k)=1$. Then $E_2$ must
intersect the last curve of the string to get blown down. Also we must
necessarily have $E_3\cdot G=0$. If $E_3$ is disjoint from $C_1\cup\cdots\cup
C_k$, then the string must have length $1$.  It follows that, after completing
the blowing down process, the intersection number of the images of $D$ and $G$
will be either $2$ or $4$, depending on whether $E_2\cdot D=0$ or $1$, a
contradiction in both cases. So we may assume that $E_3\cdot
(C_1\cup\cdots\cup C_k)=1$. Note that the only way an appropriate singularity
on the image of $D$ can arise is if $E_3\cdot D=1$. It follows that we must
have $E_3\cdot C_k=1$ and $E_2\cdot C_1=1$.  Note also that we necessarily
have $E_2\cdot F=1$, otherwise the intersection number of the images of $F$
and $G$ will not be $3$.  If $E_3\cdot F=0$, then, after completing the
blowing down process, the intersection number of the images of $F$ and $D$
will be at most $8$, a contradiction. If $E_3\cdot F=1$, then after completing
the blowing down process, the intersection number of the images of $F$ and $G$
will be $4$, again a contradiction. This last observation concludes the
discussion of Case I and shows that $E_1\cdot B_1=1$ is not possible.

\medskip
{\bf Case II:} $E_1\cdot B_2=1$. Again we consider the two
possibilities: $E_1\cdot (C_1\cup\cdots\cup C_k)=1$ or $E_1\cdot
(C_1\cup\cdots\cup C_k)=0$.

\medskip
{\bf {IIA.}} $E_1\cdot (C_1\cup\cdots\cup C_k)=1$.
Note that $E_1\cdot F=0$, otherwise when the image of $C_i$ is
eventually blown down the image of $F$ will develop more than one
singularity, a contradiction. For a similar reason we also have
$E_1\cdot D=0$. Suppose that $E_1\cdot C_i=1$. We consider the
possibilities for $i$.

\medskip
{(i)} $i=1$.  Suppose that $E_2\cdot (C_1\cup\cdots\cup C_k)=0$. Then
the condition that the image of $G$, after completing the blowing down
process, be a $(+1)$-curve forces us to have $E_3\cdot G=1$ and
$E_3\cdot (C_1\cup\cdots\cup C_k)=0$. Also, the condition that the
images of $F$ and $D$ have nodes and that the intersection numbers of
the images of $F$ and $G$, and $D$ and $G$ be $3$ forces us to have
$E_2\cdot F=2$, $E_2\cdot D=0$ and $E_3\cdot F=0$, $E_3\cdot D=2$, or
vice-versa. Finally, the condition that $F$ will have self-intersection
number $9$ forces us to have $k=3$. But then it follows that the
intersection number of the images of $F$ and $D$, after completing the
blowing down process, will be $6$, a contradiction.

Suppose that $E_2\cdot (C_1\cup\cdots\cup C_k)=1$. Then $E_2$ will
intersect the last curve of the string to get blown down. Note that
$E_2\cdot D=0$, otherwise, after completing the blowing down process,
the intersection number of the images of $D$ and $G$ will be greater
than $3$, a contradiction. Similarly $E_2\cdot F=0$. Note also that
$E_3$ is necessarily disjoint from $G$. Thus if $E_3$ is also disjoint
from the string or from $D$, it follows that the intersection number
of $D$ and $G$ after completing the blowing down process will be $2$,
a contradiction. Thus $E_3$ necessarily intersects the string and
$D$. In fact, we require that $E_3\cdot C_k=1$. Now the condition that
the image of $F$ have a singularity forces us to have $E_3\cdot
F=1$. Also, the condition that the image of $F$ have self-intersection
number $9$ forces us to have $k=3$. However, if $k=3$, then the
intersection number of the images of $F$ and $D$, after completing the
blowing down process, will be $10$, a contradiction.

\medskip
{(ii)} $1<i<k\ (k\geq 3)$.  If $E_2\cdot (C_1\cup\cdots\cup C_k)=0$,
then, as before, we require that $E_3\cdot G=1$,\ $E_3\cdot
(C_1\cup\cdots\cup C_k)=0$. It follows that we must have $k=3$,
otherwise, after completing the blowing down process, the image of $F$
will either have a singularity of multiplicity greater than two or
will have more than one singular point, neither of which is permitted
for a cubic in $\cpk$.  Now the condition that the images of $F$ and
$G$ have intersection number $9$ forces us to have $E_2\cdot
F=E_3\cdot F=1$. But then the image of $F$, after completing the
blowing down process, will have self-intersection number $10$, a
contradiction.  Thus $E_2\cdot (C_1\cup\cdots\cup C_k)=1$ and $E_2$
intersects the last curve of the string that gets blown down. Note
that, as in the previous case, $E_2\cdot F=0$, $E_2\cdot D=0$.

Suppose that $C_i\cdot C_i=-4$. Then the image of $C_i$ will be a
$(-1)$-curve, after blowing down $E_1$ and then sequentially blowing
down the images of $B_2,B_1$. It follows that the images of
$C_{i-1},C_{i+1}$ must be the last two curves of the string attached
to $D$ to get blown down. Since $E_2\cdot (C_1\cup\cdots\cup C_k)=1$,
note that, as before, we require $E_3\cdot (C_1\cup\cdots\cup C_k)=1$,
$E_3\cdot D=1$. It follows that we must have $E_3\cdot C_k=1$.  Note
that $E_3\cdot F=0$, otherwise the image of $F$ after completing the
blowing down process would have more than one singular points, a
contradiction. Now, after completing the blowing down process, we find
that the intersection number of the images of $D$ and $F$ will be $8$,
a contradiction.

Suppose that $C_i\cdot C_i<-4$. Then after blowing down $E_1$ and then
sequentially blowing down $B_2,B_1$, the image of $C_i$ will not be a
$(-1)$-curve. As before, we can show that $E_3\cdot C_k=1$, $E_3\cdot
D=1$. The condition that $F$ become singular forces us to have
$E_3\cdot F=1$. Now after completing the blowing down process we see
that the $F$ will have more than one singularity, since $i>1$, a
contradiction.

\medskip
{(iii)} $i=k\ (k\geq 2)$.  If $E_2\cdot (C_1\cup\cdots\cup C_k)=0$,
then, as before, we require that $E_3\cdot G=1$,\ $E_2\cdot
(C_1\cup\cdots\cup C_k)=0$. To obtain the correct types of
singularities on the images of $F$ and $D$ and to meet the requirement
that the intersection numbers of the images of $F$ and $G$, and $D$
and $G$, after completing the blowing down process, be $3$, we require
that $E_2\cdot F=2,\ E_3\cdot F=0$ or $E_2\cdot F=0,\ E_3\cdot F=2$
and likewise for $D$. It follows that after completing the blowing
down process the intersection number of the images of $F$ and $D$ will
be $8$, a contradiction. So $E_2\cdot (C_1\cup\cdots\cup C_k)=1$ and
$E_2$ intersects the last curve of the string that gets blown down.

Suppose that $E_3\cdot (C_1\cup\cdots\cup C_k)=0$ or $E_3\cdot
D=0$. Then since $E_3\cdot G=0$, after completing the blowing down
process the intersection number of the images of $D$ and $G$ will be
$2$, a contradiction. So $E_3\cdot (C_1\cup\cdots\cup C_k)=0$ and
$E_3\cdot D=1$. Similarly we can check that $E_3\cdot F=1$.

Suppose that $E_3\cdot C_j=1$ for $j<k$. Blow down $E_1,E_2,E_3$ and
then sequentially blown down the images of $B_2,B_1$. Note then that,
after the images of $C_k,C_{k-1},\ldots,C_j$ have been sequentially
blown down, the image of $F$ will become singular. Also after the
images of $C_j,C_{j-1},\ldots,C_2$ have been sequentially blown down
the image of $D$ will become singular. Since the images of $F$ and $D$
should have exactly one singularity, the image of $C_j$ must
necessarily be the last curve of the string to get blown down. It
follows that $j$ must be $2$. It is now easy to check that, after the
blowing down process has been completed, the intersection number of
the images of $F$ and $D$ will be $8$, a contradiction.

Suppose that $E_3\cdot C_k=1$. Then once the image of $C_k$ is blown
down the image of $F$ will become singular. It follows that $k$ must
be $1$, contrary to assumption.

\medskip
{\bf {IIB.}} $E_1\cdot (C_1\cup\cdots\cup C_k)=0$.
If $E_2\cdot (C_1\cup\cdots\cup C_k)=0$, then we must have $E_3\cdot
G=1$ and $E_3\cdot (C_1\cup\cdots\cup C_k)=0$. It follows that the string
$C_1,\ldots,C_k$ must be empty, which is never the case. So $E_2\cdot
(C_2\cup\cdots\cup C_k)=1$ and $E_2$ intersects the last curve that
gets blown down. We thus necessarily have $E_3\cdot G=0$.

Suppose that $E_1\cdot F=0$. If $E_2\cdot F=0$ also, then the only way
that the image of $F$ can have the correct type of singularity is if
$E_3\cdot F=2$ and $E_3\cdot (C_1\cup\cdots\cup C_k)=0$. But then,
after completing the blowing down process, we find that the
intersection number of the images of $F$ and $G$ will be $2$, a
contradiction.  So $E_2\cdot F=1$. There are now two ways that the
image of $F$ can have the correct type of singularity: if $E_3\cdot
F=1$ and $E_3\cdot (C_1\cup\cdots\cup C_k)=1$ or if $E_3\cdot F=2$ and
$E_3\cdot (C_1\cup\cdots\cup C_k)=0$. In the former case, after
completing the blowing down process, the intersection number of the
images of $F$ and $G$ will be $4$, a contradiction. In the latter
case, after completing the blowing down process, the intersection
number of the images of $D$ and $G$ will be either $2$ or $4$
depending on whether $E_2\cdot D=0$ or $2$, a contradiction in either
case.

Suppose that $E_1\cdot F=1$. If $E_2\cdot F=0$, then, after completing
the blowing down process, the intersection number of the images of $F$
and $G$ will be either $2$ (if $E_3\cdot F=1$ and $E_3\cdot
(C_1\cup\cdots\cup C_k)=1$) or $1$ (if $E_3\cdot F=0$ or $E_3\cdot
(C_1\cup\cdots\cup C_k)=0$), a contradiction in either case. So
$E_2\cdot F=1$. Note now that if $E_3\cdot F=1$, then the
self-intersection number of the image $F$, after completing the
blowing down process, will be greater than $9$, which is not possible
for a cubic in $\cpk$, implying that $E_3\cdot F=0$. Also if $E_2\cdot
D=1$, then, after completing the blowing down process, the
intersection number of the images of $D$ and $G$ will be greater than
$3$, a contradiction, hence we conclude that $E_2\cdot D=0$. Next note
that if $E_3\cdot (C_1\cup\cdots\cup C_k)=0$ or $E_3\cdot D=0$, then
since $E_3\cdot G=0$, after completing the blowing down process, the
intersection number of the images of $D$ and $G$ will be $2$, a
contradiction. So $E_3\cdot (C_1\cup\cdots\cup C_k)=1$ and $E_3\cdot
D=1$. It follows that we must have $E_3\cdot C_k=1$ and $E_2\cdot
C_1=1$. Also if $E_1\cdot D=0$, then, after completing the blowing
down process, the intersection number of the images of $D$ and $F$
will be $5$, a contradiction. So we must have $E_1\cdot D=1$. Thus the
three $(-1)$-curves $E_1,E_2,E_3$ must be as given by the Proposition.
\end{proof}

\subsection{The family $\fra $}
Finally we consider four-legged graphs in the family $\fra$.
\begin{figure}[ht]
\begin{center}
\includegraphics[width=10cm]{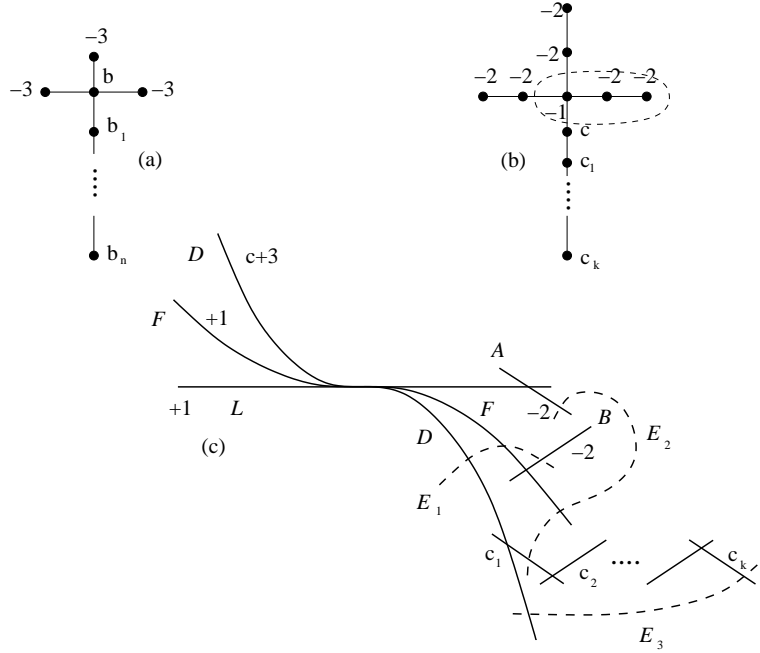}
\end{center}
\caption{The four-legged graphs in $\fra$. Proposition~\ref{p:a4leg}
  allows another configuration for $E_2$ in (c), where it intersects
  $C_2$ instead of $C_1$.}
\label{f:a4leg}
\end{figure}
The generic four-legged member $\G$ of $\fra$ is given in
Figure~\ref{f:a4leg}(a) with the dual graph in (b).  After three
blow-downs we obtain the configuration $K$ indicated in
Figure~\ref{f:a4leg}(c). Suppose that $Z$ is the closed symplectic
$4$-manifold we get by symplectically gluing the compactifying divisor
$W_{\G '}$ (containing $K$) to a weak symplectic $\bfq$HD filling of
$Y_\G$. Then it is easy to check that there must be three
$(-1)$-curves, say $E_1,E_2,E_3$, not contained in the string
$C_1,\ldots,C_k$, such that, after blowing down these three
$(-1)$-curves, the image of $B$ can be blown down and the images of
the curves in the string $C_1,\ldots,C_k$ can be sequentially blown
down so that in the process $F$ and $D$ are transformed to a pair of
cubics in $\cpk$ and the images of $L$ and $A$ are lines. Since in the
blowing down process $B$ will be eventually transformed into a curve
which can be blown down, one of the three $(-1)$-curves, call it
$E_1$, must intersect $B$.

\begin{prop}\label{p:a4leg}
  If a $\bfq$HD filling exists in the situation described above, then
  either $\G '$ blows down to a 3-legged graph (and was treated
  earlier), or $E_1$ intersects $D$, $F$ and $B$, $E_2$ intersects
  $A$, $F$ and either $C_1$ or $C_2$ and $E_3$ intersects $D$ and
  $C_k$. The corresponding framings in the latter case are given as
  $c=-k$, $c_2=-3$ and $c_1=c_3=\ldots =c_k=-2$. In particular, the
  corresponding resolution graph is of the form given by
  Figure~\ref{f:qhd4uj}(a).
\end{prop}
\begin{proof}
Note that if $E_1\cdot A=1$, then $E_1\cdot (C_1\cup\cdots\cup
C_k)=0$, otherwise after blowing down $E_1$ and then the image of
$B$, the image of $A$ will become singular when the image of $C_i$ is
eventually blown down, where $E_1\cdot C_i=1$, which contradicts the
fact that the image of $A$ in $\cpk$ will be a line. Thus at least
one $(-1)$-curve different from $E_1$ should intersect $A$. Let us
call this $(-1)$-curve $E_2$.
We now begin the case-by-case analysis.

\medskip
{\bf Case I:} $E_1\cdot (C_1\cup\cdots\cup C_k)=1$. Suppose that
$E_1\cdot C_i=1$. In this case, by the argument above, we will
necessarily have $E_1\cdot A=0$. Note that if $E_1\cdot F=1$, then
after $E_1$ and the image of $B$ are blown down, the image $F^\prime$
of $F$ will be singular. However, the intersection number of the image
$C_i^\prime$ of $C_i$ and $F^\prime$ will be $2$. Thus when the image
of $C_i^\prime$ is eventually blown down the image of $F^\prime$ have
a second singularity, which contradicts the fact that it must
eventually blow down to a cubic in $\cpk$. Thus $E_1\cdot F=0$. Also,
we must have $E_1\cdot D=0$, otherwise, after repeatedly blowing down,
the image of $D$ will eventually have a triple point, which
contradicts the fact that the image of $D$ in $\cpk$ should also be a
cubic.

Note that if $E_2\cdot (C_1\cup\cdots\cup C_k)=0$, then we must have
have $E_3\cdot A=1$ and $E_3\cdot (C_1\cup\cdots\cup C_k)=1$, since,
after completing the blowing down process, the image of $A$ should be
a smooth curve of self-intersection number $1$. Renumbering $E_2$ and
$E_3$, if necessary, we may assume that $E_2\cdot (C_1\cup\cdots\cup
C_k)=1$. Suppose that $E_2\cdot C_j=1$.  Notice that, in the blowing
down process, the image of $C_j$ must either be the last curve of the
string attached to $D$ to get blown down or it must be the penultimate
curve to get blown down, since otherwise, after the blowing process is
complete, the self-intersection number of the image of $A$ will be
greater than $1$, a contradiction.

\medskip
{(i)} $i=1$.

\medskip {(ia)} Suppose first that the image $C_j$ is the last curve
of the string to get blown down. Then we must have $E_3\cdot A=1$, and
hence $E_3\cdot (C_1\cup\cdots\cup C_k)=0$.  Now, since $E_1\cdot
D=0$, there are two ways that an appropriate singularity can appear on
image of $D$: either $E_2\cdot D=1$ or $E_3\cdot D=2$.  Suppose that
$E_2\cdot D=1$. Then $E_3\cdot D=0$, otherwise, after completing the
blowing down process, the intersection number of the images of $D$ and
$A$ would be greater than $3$, a contradiction. We now have $E_2\cdot
F=0$ or $1$.  If $E_2\cdot F=0$, then we must have $E_3\cdot F=2$,
otherwise the image of $F$, after completing the blowing down process,
would be smooth and rational, which is a contradiction. Now the
condition that the self-intersection number of the image of $F$, after
completing the blowing down process, will be $9$, forces us to have
$k=2$. But then, after completing the blowing down process, the
intersection number of the images of $F$ and $D$ will be $7$, a
contradiction. If $E_2\cdot F=1$, then $E_3\cdot F=0$, otherwise the
intersection number of the images of $F$ and $A$, after completing the
blowing down process, would be greater than $3$, a contradiction. Now,
again, the condition that the self-intersection number of the image of
$F$, after completing the blowing down process, will be $9$, forces us
to have $k=3$. But then, after completing the blowing down process,
the intersection number of the images of $F$ and $D$ will be $10$,
again a contradiction.

Suppose that $E_3\cdot D=2$. Then $E_2\cdot D=0$. We now have
$E_2\cdot F=0$ or $1$. If $E_2\cdot F=0$, then we must have $E_3\cdot
F=2$. Now, as before, the condition that the self-intersection number
of the image of $F$, after completing the blowing down process, will
be $9$, forces $k=3$. But then, after completing the blowing down
process, the intersection number of the images of $F$ and $D$ will be
$10$, a contradiction. If $E_2\cdot F=1$, then we must have $E_3\cdot
F=0$. Thus, again, the condition that the self-intersection number of
the image of $F$, after completing the blowing down process, will be
$9$, forces $k=3$. And, this time, after completing the blowing down
process, the intersection number of the images of $F$ and $D$ will be
$7$, again a contradiction.

\medskip {(ib)} The image of $C_j$ is then the penultimate curve of the string
to get blown down. Then we must have $E_3\cdot A=0$. Also, we must have
$E_2\cdot D=0$, otherwise, after completing the blowing down process, the
intersection number of the images of $D$ and $A$ would be greater than $3$, a
contradiction. Similarly, we must have $E_2\cdot F=0$.

Suppose that $E_3\cdot (C_1\cup\cdots\cup C_k)=0$ or $E_3\cdot
D=0$. Then, after completing the blowing down process, the
intersection number of the images of $D$ and $A$ will be at most $2$,
a contradiction. So $E_3\cdot (C_1\cup\cdots\cup C_k)=1$ and $E_3\cdot
D=1$. If $E_3\cdot C_l=1$ for $l<k$, then we must have $l=k-1$ and
$j=k$, otherwise, after completing the blowing down process, the image
of $D$ will have more than one singular point, a
contradiction. However, if $l=k-1$ and $j=k$, then, after completing
the blowing down process, the intersection number of the images of $D$
and $A$ will be $2$, a contradiction. So we must have $E_3\cdot
C_k=1$. Also we must have $E_3\cdot F=1$, otherwise the image of $F$,
after completing the blowing down process will be smooth, a
contradiction. Now, the condition that the self-intersection number of
the image of $F$, after completing the blowing down process, will be
$9$, forces us to have $k=3$. But then, after completing the blowing
down process, the intersection number of the images of $F$ and $D$
will be $10$, a contradiction.

\medskip
{(ii)} $1<i<k\ (k\geq 3)$.

\medskip
{(iia)} The image of $C_j$ is last curve of the string to get blown
down. Then we must have $E_3\cdot A=1$, and hence $E_3\cdot
(C_1\cup\cdots\cup C_k)=0$. If $k$ is greater than $2$, then, after
completing the blowing down process, the image of $F$ will either have
a point of multiplicity greater than $2$ or have more than one
singular point, neither of which is possible for a cubic in $\cpk
$. Thus we must have $k=3$ and thus $j=1$ or $3$. Also, we must have
$E_2\cdot F=0$, otherwise, after completing the blowing down process,
the image of $F$ will have a triple point, a
contradiction. Furthermore, we must have $E_3\cdot F=1$, otherwise,
after completing the blowing down process, the intersection number of
the images of $F$ and $A$ will be less that $3$, a contradiction.  Now
the only way a singularity of the appropriate type can appear on the
image of $D$ is if $E_2\cdot D=1$ or $E_3\cdot D=2$.

Suppose first that $E_2\cdot D=1$. Then we must have $E_3\cdot D=0$,
otherwise, after completing the blowing down process, the intersection
number of the images of $A$ and $D$ will be greater than $3$, a
contradiction. It follows that, after completing the blowing down
process, the intersection number of the images of $F$ and $D$ will be
at most $8$, which contradicts the fact that images of $F$ and $D$ in
$\cpk$ are a pair of cubics.

Suppose now that $E_3\cdot D=2$. Then we must have $E_2\cdot D=0$,
otherwise, after completing the blowing down process, the intersection
number of the images of $D$ and $A$ will be greater that $3$, a
contradiction. It follows that, after completing the blowing down process,
the intersection number of the images of $F$ and $D$ will be at most
$8$, a contradiction.

\medskip
{(iib)} The image of $C_j$ is the penultimate curve of the string to
get blown down. Then we must have $E_3\cdot A=0$ and $E_2\cdot
D=E_2\cdot F=0$. Also, we must have $E_3\cdot (C_1\cup\cdots\cup
C_k)=1$ and $E_3\cdot D=1$.

Suppose that $E_3\cdot C_l=1$ for $l<k$. Then the image of $C_l$ must
be the last curve of the string attached to $D$ to get blown
down. Indeed, it is easy to see that after the image of $C_l$ is blown
down, the image of the portion $C_{l+1},\ldots,C_k$ of the the string
must be a point, otherwise, after completing the blowing down process,
the image of $D$ will have more than one singular point. Thus we must
have $i>l$ or $j>l$. In the former case, after the portion
$C_{l},\ldots,C_k$ of the the string has been collapsed to a point,
the image of $F$ will be singular and thus the image of $C_l$ must be the
last curve of the string to get blown down. In the latter case, since
the image of $C_j$ is the penultimate curve of the string to get blown
down, $C_l$ must be the last curve of the string to get blown down. Now
again using the assumption that the image of $C_j$ is the penultimate
curve of the string to get blown down, we must have either $j<l$ or
$j>l$. Suppose that $j<l$. Then we must have $i>l$. Also, we must have
$E_3\cdot F=0$, otherwise, after completing the blowing down process,
the image of $F$ will have a singularity of multiplicity greater than
$2$, a contradiction. Now, after completing the blowing down process,
the intersection number of the images of $A$ and $F$ will be $2$, a
contradiction. Suppose that $j>l$. Then, after completing the blowing
down process, the intersection number of the images of $A$ and $D$
will be $2$, again a contradiction.

Suppose that $E_3\cdot C_k=1$. Then we must have $E_3\cdot F=0$ or
$1$. Suppose that $E_3\cdot F=0$. Then, in the blowing down process,
the images of the curves $C_{i-1}$ and $C_{i+1}$ must be the last two
curves of the string attached to $D$ to get blown down. It follows
that we must have $i=2$. It is now easy to check that, after
completing the blowing down process, the image of $F$ will have
self-intersection number $8$, a contradiction. Suppose that $E_3\cdot
F=1$. Then the image of $C_i$ must be the last curve of the string to
get blown down. It follows that we must have $i=2$ and $j=1$. We now
find that, after completing the blowing down process, the intersection
number of the images of $F$ and $A$ will be $2$, a contradiction.

\medskip
{(iii)} $i=k\ (k\geq 2)$.

\medskip
{(iiia)} The image of $C_j$ is last curve of the string to get blown
down. Then we must have $E_3\cdot A=1$, and hence $E_3\cdot
(C_1\cup\cdots\cup C_k)=0$. Also we must have $j=1$. Now the only way
a singularity of the appropriate type can appear on the image of $D$
is if $E_2\cdot D=1$ or $E_3\cdot D=2$.

Suppose that $E_2\cdot D=1$. Then we must have $E_3\cdot D=0$. Now we
have $E_2\cdot F=0$ or $1$. If $E_2\cdot F=0$, then it is easy to see
that, after completing the blowing down process, the intersection
number of the images of $F$ and $D$ will be $5$, a contradiction.  If
$E_2\cdot F=1$, then one can check that, after completing the blowing
down process, the intersection number of the images of $F$ and $D$
will be $8$, again a contradiction.

Suppose that $E_3\cdot D=2$. Then we must have $E_2\cdot D=0$. Again
we have $E_2\cdot F=0$ or $1$. If $E_2\cdot F=0$, then we must have
$E_3\cdot F=2$. It follows that, after completing the blowing down
process, the intersection number of the images of $F$ and $D$ will be
$8$, a contradiction. If $E_2\cdot F=1$, then we must have $E_3\cdot
F=0$. In this case, after completing the blowing down process, the
intersection number of the images of $F$ and $D$ will be $5$, again a
contradiction.

\medskip
{(iiib)} The image of $C_j$ is the penultimate curve of the string to
get blown down. Then we must have $E_3\cdot A=0$ and $E_2\cdot
D=E_2\cdot F=0$. Also, we must have $E_3\cdot (C_1\cup\cdots\cup
C_k)=1$ and $E_3\cdot D=1$. Furthermore, we must have $E_3\cdot F=1$,
otherwise, after completing the blowing down process, the image of $F$
would be smooth, a contradiction. Now note that if $l\neq 1$, then we
must have $l=2$ and $j=1$, otherwise, after completing the blowing
down process, the image of $F$ will have more than one singular point,
a contradiction. If $l=1$, then $C_1$ must be the last curve to get
blown down, otherwise, after completing the blowing down process, the
image of $D$ will have more than one singular point, a
contradiction. Thus we must have $j=2$. It now follows that, after
completing the blowing down process, the intersection number of the
images of $D$ and $A$ will be $2$, a contradiction. If $l=2$ and $j=1$,
then $C_2$ must be the last curve to get blown down and in this case
it follows that, after completing the blowing down process, the
intersection number of the images of $F$ and $A$ will be $2$, again a
contradiction.

\medskip
{\bf Case II:} $E_1\cdot (C_1\cup\cdots\cup C_k)=0$.

\medskip
{\bf {IIA.}} $E_1\cdot A=1$. Since we are assuming that $E_2\cdot
A=1$ also, we will necessarily have $E_2\cdot (C_1\cup\cdots\cup
C_k)=0$ and $E_3\cdot A=0$. Also, since the string $C_1\,\ldots,C_k$
is nonempty for every 4-legged graph $\G$ in $\fra$, we must have
$E_3\cdot (C_1\cup\cdots\cup C_k)=1$. Now if $E_1\cdot D=0$, then,
after completing the blowing down process, the intersection number of
the images of $D$ and $A$ will be at most $2$, a contradiction. It
follows that we must have $E_1\cdot D=1$ and thus we must also have
$E_2\cdot D=1$.

Suppose that $E_1\cdot F=1$. Then we must have $E_2\cdot F=0$. If
$E_3\cdot F=0$ also holds, then, after completing the blowing down
process, the self-intersection number of the image of $F$ will be $6$,
a contradiction. So we must have $E_3\cdot F=1$ and $k$ must be
$2$. But then, after completing the blowing down process, the
intersection number of the images of $F$ and $D$ will be $10$, a
contradiction.

Suppose that $E_1\cdot F=0$. Then we must have $E_2\cdot F=2$. Again
we require $E_3\cdot F=1$ and $k=2$. It thus follows again that, after
completing the blowing down process, the intersection number of the
images of $F$ and $D$ will be $10$, a contradiction as before.

\medskip
{\bf {IIB.}} $E_1\cdot A=0$. We may now assume $E_2\cdot
(C_1\cup\cdots\cup C_k)=1$, since if $E_2\cdot (C_1\cup\cdots\cup
C_k)=0$, then we would necessarily have $E_3\cdot A=1$ and $E_3\cdot
(C_1\cup\cdots\cup C_k)=1$, and we would just renumber the
$(-1)$-curves. Suppose that $E_2\cdot C_j=1$. It follows that, in the
blowing down process, the image of $C_j$ is either the last curve of
the string to get blown down or the penultimate curve to get blown
down.

\medskip
{(i)} Suppose first that the image of $C_j$ is last curve of the
string to get blown down.  Then we must have $E_3\cdot A=1$ and
$E_3\cdot (C_1\cup\cdots\cup C_k)=0$. Since we are assuming that
$E_1\cdot A=0$, we must have that $k=1$. Now if $E_2\cdot F=0$, then,
after completing the blowing down process, the intersection number of
the images of $A$ and $F$ will be at most $2$, a contradiction. So
$E_2\cdot F=1$ and thus $E_3\cdot F=1$ also. It follows that we must
have $E_1\cdot F=1$, otherwise, after completing the blowing down
process, the image of $F$ would be smooth or have more than one
singularity, a contradiction in both cases.

Suppose that $E_2\cdot D=1$. Then we must have $E_3\cdot D=0$. Note
also that we must have $E_1\cdot D=1$, otherwise, after completing the
blowing down process, the intersection number of the images of $F$ and
$D$ will be different from $9$, a contradiction. It follows that $D$
must have self-intersection number $2$ and $C_1$ must have
self-intersection number $-2$. It is easy to see that in this case
$\G$ is just the unique three-legged graph in the family $\fra$ with
four vertices and we already know that in this case the corresponding
contact $3$-manifold $(Y_\G, \xi _{\G})$ admits a $\bfq$HD filling.

Suppose that $E_2\cdot D=0$. Then we must have $E_3\cdot D=2$. Again
we can check that we must have $E_1\cdot D=1$. As in the previous
case, it follows that $D$ must have self-intersection number $2$ and
$C_1$ must have self-intersection number $-2$, and this case has
already been considered.

\medskip {(ii)} The image of $C_j$ is the penultimate curve of the
string to get blown down. Then we must have $E_3\cdot A=0$. Note that
if $E_2\cdot D=1$, then, after completing the blowing down process,
the intersection number of the images of $A$ and $D$ will be $4$, a
contradiction. Thus $E_2\cdot D=0$.  Also we must have $E_3\cdot
(C_1\cup\cdots\cup C_k)=1$ and $E_3\cdot D=1$, otherwise, after
completing the blowing down process, the intersection number of the
images of $A$ and $D$ will be at most $2$, a contradiction. Suppose
that $E_3\cdot C_l=1$. Now if $l<k$, then we must have $k=2$, $l=1$
and $j=2$.  But then, after completing the blowing down process, the
intersection number of the images of $A$ and $D$ will be $2$, a
contradiction. So $l=k$.  It follows that we must have $j=1$ or
$2$. Now note that if $E_2\cdot F=0$, then, after completing the
blowing down process, the intersection number of the images of $A$ and
$F$ will be at most $2$, a contradiction. So we must have $E_2\cdot
F=1$. It also follows that we must have $E_3\cdot F=0$, otherwise,
after completing the blowing down process, the intersection number of
the images of $A$ and $F$ will be greater than $3$, a
contradiction. We now must have $E_1\cdot F=1$, otherwise, after
completing the blowing down process, the image of $F$ will be smooth,
a contradiction. For each value of $k$ and for $j=1,2$, the blowing
down process now fixes $c,c_1,\ldots,c_k$, as stated in the
Proposition.
\end{proof}
Now we are ready to give the proof of the second main result of the
paper.
\begin{proof}[Proof of Theorem~\ref{t:main2}]
Consider a spherical Seifert singularity $S_{\G}$ with minimal good
resolution graph having at least four legs (and central framing $<
-2$). Once again, the existence of a $\bfq$HD smoothing implies the
existence of a $\bfq$HD filling of the Milnor fillable contact
structure $\xi _{\G}$ on the link $Y_{\G}$ showing the implication
$(1)\Rightarrow (2)$.  Suppose now that $(Y_{\G }, \xi _{\G})$ admits
a $\bfq$HD filling. By Theorem~\ref{t:ssw}, we get that $\G$ is a
4-legged graph in $\fra\cup \frb \cup \frc$. Therefore the combination
of Propositions~\ref{p:c4leg}, \ref{p:b4leg} and \ref{p:a4leg} implies
$(2)\Rightarrow (3)$. Finally $(3)\Rightarrow (1)$ follows from the
results of \cite[Examples~8.7, 8.12 and 8.14]{SSW}, cf. also
\cite{Wahluj}. These existence results then conclude the proof of the
theorem.
\end{proof}

For the sake of completeness we provide curve configurations in $\cpk$
with the property that repeated blow-ups provide the configurations of
Figures~\ref{f:c4leg}(c), \ref{f:b4leg}(c) and \ref{f:a4leg}(c), and hence,
by Pinkham's result \cite{Pink} (as formulated in
\cite[Theorem~8.1]{SSW}), we get an alternative proof of the existence
of $\bfq$HD smoothings.  Below we will restrict ourselves to the
description of the curves and their intersection patterns, and leave
it to the reader to check that an appropriate blow-up sequence
restores the diagrams listed above.

Let $D$ be the cubic curve
defined by the equation
\[
f(x,y,z) =y^2z-x^3-x^2z
\]
and $L$ the line $\{ z=0\}$.

In order to find a configuration which can be blown up to
Figure~\ref{f:c4leg}(c), let us add the cubic $D_1$ given by the
equation
\[
f_1(x,y,z)=y^2z + \frac 12 xyz + yz^2 - \frac 98 x^3 -
2x^2z - xz^2
\]
to $L$ and $D$.  The curves $D$ and $D_1$ are rational nodal cubics with nodes
at $[0:0:1]$ and $[-\frac 23:-\frac 13:1]$, respectively. It is easy to check
that both $D$ and $D_1$ are triply tangent to $L$ at the point $[0:1:0]$ and
are also triply tangent to each other at $[0:1:0]$ and have intersection
multiplicity $6$ at the point $[0:0:1]$.

For finding the configuration providing the base for
Figure~\ref{f:b4leg}(c), we consider $L$ and $D$ as before, together
with $D_2$ given by the equation
\[
f_2(x,y,z)=y^2z+2xyz+2yz^2-2x^3-4x^2z-2xz^2.
\]
The curve $D_2$ is a rational nodal cubic with a node at $[-1:0:1]$, and $L$, $D$
and $D_2$ are pairwise triply tangent at $[0:1:0]$. Also $D$ and $D_2$
intersect at $[0:0:1]$ with intersection multiplicity $4$ and at $[-1:0:1]$
with intersection multiplicity $2$.  Consider, furthermore, $L_1$ given by the
equation $\{ x+z=0\}$. It passes through the point $[0:1:0]$ and is tangent to
$D$ at $[-1:0:1]$.

Finally we describe a configuration from which repeated blow-ups
result in the configuration of Figure~\ref{f:a4leg}(c). Once again,
consider $L$ and $D$ as before, together with the cubic $D_3$ given by
the equation $f_3(x,y,z)=$
\[
y^2z+(1-i\sqrt{3})xyz+\frac 49(3-i\sqrt{3})yz^2+\frac
12(-1+i\sqrt{3})x^3+(-2+i\sqrt{3})x^2z-\frac 49(-3+i\sqrt{3})xz^2.
\]
This curve is a rational nodal cubic with a node at $[-\frac 43:-\frac 49
i\sqrt{3}:1]$. The line $L$ and the curves $D$, $D_3$ are pairwise triply
tangent at $[0:1:0]$. Also the curves $D$ and $D_3$ intersect at each of the
points $[0:0:1]$ and $[-\frac 43:-\frac 49 i\sqrt{3}:1]$ with intersection
multiplicity $3$.  Let $N$ be the line $\{y-i\sqrt{3}(x+\frac 89z)=0\}$; it is
triply tangent to $D$ at $[-\frac 43:-\frac 49 i\sqrt{3}:1]$ and intersects
$D_3$ at the same point with intersection multiplicity $3$.

\end{document}